\documentclass[10pt]{amsart}
\usepackage{pxfonts}
\usepackage{mathrsfs}
\usepackage[OT2,T1]{fontenc}
\usepackage[
textwidth=14.5cm, 
textheight=21cm,
hmarginratio=1:1,
vmarginratio=1:1]{geometry}

\usepackage{pifont}

\DeclareSymbolFont{cyrletters}{OT2}{wncyr}{m}{n}
\DeclareMathSymbol{\Lfun}{\beta}{cyrletters}{"62}
\DeclareMathSymbol{\Rfun}{\beta}{cyrletters}{"76}

\numberwithin{equation}{subsection}
\numberwithin{figure}{subsection}

\usepackage{amssymb} 
\usepackage{amsmath}
\usepackage{amsthm}
\usepackage{epstopdf}
\usepackage{wrapfig}
\usepackage{marvosym}
\usepackage{tikz}
\usepackage{verbatim}
\usepackage{mathrsfs}
\newcommand{\bp}{{\boldsymbol \pi}}
\newcommand{\ba}{{\boldsymbol \alpha}}
\newcommand{\bb}{{\boldsymbol \beta}}
\newcommand{\bg}{{\boldsymbol \gamma}}
\newcommand{\bd}{{\boldsymbol \delta}}

\newcommand{\Aop}{{\mathbf{A}}}

\renewcommand{\phi}{\varphi}

\newcommand{\Z}{\mathbb{Z}}

\newcommand{\R}{\mathbb{R}}
\newcommand{\Te}{\mathbb{T}}
\newcommand{\e}{\mathrm{e}}
\newcommand{\diff}{\mathrm{d}}

\renewcommand{\and}{\hspace{8pt} \text{and} \hspace{8pt}}

\newcommand{\dd}{\,d}
\newcommand{\nl}{\left\lVert}
\newcommand{\nr}{\right\rVert}

\newcommand{\M}{\mathcal{M}}
\newcommand{\wtM}{\mathcal{N}}
\newcommand{\twtM}{N}

\renewcommand{\S}{\mathscr{S}}

\newcommand{\supp}{\operatorname{supp}}

\newcommand{\calC}{\mathcal{C}}

\DeclareMathOperator{\id}{{id}}

\numberwithin{equation}{section}
\newtheorem{thm}{Theorem}[subsection]
\newtheorem{prop}[thm]{Proposition}
\newtheorem{lem}[thm]{Lemma}

\newtheorem*{thm-others}{Theorem}
\theoremstyle{definition}
\newtheorem{defn}[thm]{Definition}

\theoremstyle{remark}
\newtheorem{rem}[thm]{Remark}
\numberwithin{equation}{subsection}
\numberwithin{figure}{subsection}

\def\R{{\mathbb R}}

\renewcommand{\descriptionlabel}[1]%
         {\dblue{#1:}\\}

\begin{document}

\title[A critical topology for $L^p$-Carleman classes with $0<p<1$]
{A critical topology for $L^p$-Carleman
\\ 
classes with $0<p<1$}
\author{Haakan Hedenmalm} 
\address{Department of Mathematics, KTH Royal Institute of Technology,
\newline Stockholm, 100 44, Sweden}
\email{haakanh@math.kth.se}
\author{Aron Wennman}
\address{Department of Mathematics, KTH Royal Institute of Technology,
\newline Stockholm, 100 44, Sweden}
\email{aronw@math.kth.se}
\date{\today}

\begin{abstract}
In this paper,  we explain a sharp phase transition phenomenon which occurs
for $L^p$-Carleman classes with exponents $0<p<1$. In principle, these 
classes are defined as usual, only the traditional $L^\infty$-bounds are 
replaced by corresponding $L^p$-bounds. To set things up rigorously, we work
with the quasinorms 
\[
\nl u\nr_{p,\M}=\sup_{n\geq 0}\frac{\lVert u^{(n)}\rVert_p}{M_n},
\] 
for some weight sequence $\M=\{M_n\}_n$ of positive real numbers, and 
consider as the corresponding $L^p$-Carleman space the completion of a 
given collection of smooth test functions. 
%Normally, we would also 
To mirror the classical definition, we add the feature of dilatation invariance 
as well, and consider a larger soft-topology space, the $L^p$-Carleman class.
%, but this can always be done at the last moment. 
A particular degenerate instance is 
when $M_n=1$ for $0\le n\le k$ and $M_n=+\infty$ for $n>k$. This would 
give the $L^p$-Sobolev spaces, which were analyzed by Peetre, following
an initial insight by Douady. 
Peetre found that these $L^p$-Sobolev spaces are highly degenerate for 
$0<p<1$. Indeed, the canonical map $W^{k,p}\to L^p$ fails to be injective, 
and there is even an isomorphism
\[
W^{k,p}\cong L^p\oplus L^p\oplus\ldots \oplus L^p,
\]
corresponding to the canonical map $f\mapsto (f,f',\ldots,f^{(k)})$ acting 
on the test functions. This means that e.g. the function and its derivative 
lose contact with each other (they ``disconnect''). Here, we analyze this 
degeneracy for the more general 
$L^p$-Carleman classes defined by a weight sequence $\M$. If $\M$ has some 
reasonable growth and regularity properties, and if the given collection 
of test functions is what we call $(p,\theta)$-\emph{tame}, 
then we find that there 
is a sharp boundary, defined in terms of the weight $\M$: on the one side, 
we get Douady-Peetre's phenomenon of ``disconnexion'', while on the other, 
the completion of the test functions consists of $C^\infty$-smooth 
functions and the canonical map $f\mapsto (f,f',f'',\ldots)$ is 
correspondingly well-behaved in the completion.  We also look at the 
more standard second phase transition, between non-quasianalyticity and 
quasianalyticity, in the $L^p$ setting, with $0<p<1$.
\end{abstract}
\maketitle

\section*{Introduction}
\subsection{The Sobolev spaces for $0<p<1$}
We first survey the properties of the Sobolev spaces with exponents in the 
range $0<p<1$.
These were considered by Jaak Peetre \cite{Peetre} after Adrien Douady 
suggested that they behave very differently from the standard 
$1\le p\le+\infty$ case. For an exponent $p$ with $0<p<1$, and an integer 
$k\geq 0$,
\emph{we define the Sobolev space $W^{k,p}(\R)$ on the real line $\R$ to be the 
abstract completion of the space $C^{k}_0$ of compactly supported $k$ times 
continuously differentiable functions, with respect to the quasinorm} 
\begin{equation}\label{sobolevnorm}
\lVert f \rVert_{k,p}=\left(\lVert f\rVert_p^p + \lVert f'\rVert_p^p+\ldots 
\lVert f^{(k)}\rVert_p^p\right)^{1/p}.
\end{equation}
Here, as a matter of notation, $\|\cdot\|_p$ denotes the quasinorm of $L^p$. 
Defined in this manner, $W^{k,p}$ becomes a quasi-Banach space whose elements 
consist of equivalence classes of Cauchy sequences of test functions.
At first glance, this definition is very natural, and other approaches 
to define the same object in the classical case $p\geq 1$ do not 
generalize. For instance, we cannot use the usual notion of weak derivatives 
to define these spaces, as functions in $L^p$ need not be locally $L^1$ for 
$0<p<1$. 

Some observations suggest that things are not the way we might expect them 
to be. For instance, as a consequence of the failure of local integrability,
for a given point $a\in\R$, the formula for the primitive
\[
F(x)=\int_{a}^{x}f(t)\,\diff t
\] 
cannot be expected to make sense.
A related difficulty is the invisibility of Dirac ``point masses'' in the 
quasinorm of $L^p$. Indeed, if $0<p<1$ and 
$u_\epsilon:=\epsilon^{-1}1_{[0,\epsilon]}$, then as $\epsilon\to 0^+$, 
we have the convergence $u_\epsilon\to0$ in $L^p$ while $u_\epsilon\to\delta_0$ 
in the sense of distribution theory. Here, we use standard notational 
conventions: $\delta_0$ denotes the unit point mass at $0$, and $1_E$ is the 
characteristic function of the subset $E$,  which equals $1$ on $E$ and 
vanishes off $E$. 

\subsection{Independence of the derivatives in Sobolev spaces}
For $1\le p\le+\infty$, we think of the Sobolev space $W^{k,p}$ as a subspace 
of $L^p$, consisting of functions of a specified degree of smoothness.
As such, the identity mapping $\ba:W^{k,p}\to L^p$ defines a canonical 
injection. 
Douady observed that we cannot have this picture in mind when $0<p<1$, as 
the corresponding canonical map $\ba$ fails to be injective.
Peetre built on Douady's observation and showed that this uncoupling (or 
disconnexion) between the derivatives goes even deeper. 
In fact, the standard map $f\mapsto (f,f',\ldots,f^{(k)})$ on test functions 
$f$ defines a topological isomorphism of the completion $W^{k,p}$ onto the 
direct sum of $k+1$ copies of $L^p$. For convenience, we state the 
theorem here.

\begin{thm}[Peetre, 1975]\label{Peetre}
Let $0<p<1$ and $k=1,2,3,\ldots$. Then $W^{k,p}$ is isometrically isomorphic to 
$k+1$ copies of $L^p$:
\begin{equation}\label{isom}
W^{k,p}\cong L^p\oplus L^p\oplus\ldots\oplus L^p.
\end{equation}
\end{thm}

This decoupling occurs as a result of the availability of approximate 
point masses which are barely visible in the quasinorm. 
As a consequence, if we define Sobolev spaces as completions with respect to 
the Sobolev quasinorm, we obtain highly pathological (and rather useless) 
objects. 

\subsection{A bootstrap argument to control the $L^\infty$-norm
in terms the $L^p$-quasinorms of the higher derivatives}

From the Douady-Peetre analysis of the Sobolev spaces $W^{k,p}$ for exponents
$0<p<1$, we might be inclined to believe that for such small $p$, we always 
run into pathology. 
However, there is in fact an argument of bootstrap type which can save the 
situation if we simultaneously control \emph{the derivatives of all orders}.
To explain the bootstrap argument, we take a (weight) sequence 
$\mathcal{M}=\{M_k\}_{k=0}^{+\infty}$ of positive reals, and define the quasinorm
\begin{equation}
\label{M-norm}
\| f\|_{p,\M}=\sup_{k\geq0}\frac{\| f^{(k)}\|_p}{M_k},\qquad f\in \S,
\end{equation}
on some appropriate linear space $\S\subset C^\infty(\R)$ 
of test functions that we choose to begin with. In the analogous setting 
of periodic 
functions (i.e., with the unit circle $\Te\cong\R/\Z$ in place of $\R$), 
it would be natural to work with the linear span of the periodic complex 
exponentials as $\S$. 
In the present setting of the line $\R$, there is no such canonical choice. 
Short of a natural explicit linear space of functions, we ask instead that 
$\S$ should satisfy a property.

\begin{defn}
($0<p\le1,\,\theta\in\R$)
We say that $f\in C^\infty(\R)$ is
\emph{$(p,\theta)$-tame} if $f^{(n)}\in L^\infty(\R)$ for 
each $n=0,1,2,\ldots$, and
\begin{equation}
\label{eq-tameness}
\limsup_{n\to+\infty}\,(1-p)^n\log\|f^{(n)}\|_\infty\leq \theta.
\end{equation}
Moreover, for subsets $\S\subset C^\infty(\R)$, we say that
$\S$ is \emph{$(p,\theta)$-tame} if every element 
$f\in\S$ is
$(p,\theta)$-tame.
\end{defn}

\begin{rem}
(a) For $\theta<0$, only the trivial function $f=0$ is $(p,\theta)$-tame, 
unless if $p=1$, in which no $(p,\theta)$-tame function exists.
For this reason, in the sequel, \emph{we shall restrict our attention 
to $\theta\ge0$}.
 
\noindent{(b)} Loosely speaking, for $\theta\ge0$, the requirement 
\eqref{eq-tameness} asks that the $L^\infty$-norms of the higher order 
derivatives do not grow too wildly.
We note that for $p$ close to one, $(p,\theta)$-tameness is a very weak 
requirement; indeed, at the endpoint value $p=1$, it is void.
\end{rem}

A natural suitable choice of a $(p,\theta)$-tame collection of test 
functions might be the following Hermite class: 
\[
\S^{\mathrm{Her}}:=
\left\{f:\,\,\,f(x)=\e^{-x^2}q(x),\,\text{ where }\, q\,\,\, 
\text{is a polynomial}
\right\}.
\]
Indeed, a rather elementary argument shows that \eqref{eq-tameness} 
holds with $\theta=0$ for $\S=\S^{\mathrm{Her}}$ 
(for the details, see Lemma~\ref{Lemma-Hermite} below). 
However, it might be the case that not all $f\in\S^{\mathrm{Her}}$ have 
finite quasinorm $\| f\|_{p,\M}<+\infty$ (this depends on the choice of weight 
sequence $\mathcal{M}$). In that case, we would then replace 
$\S^{\mathrm{Her}}$ by its linear subspace
\[
\S^{\mathrm{Her}}_{p,\M}:=\big\{f\in\S^{\mathrm{Her}}:\,
\| f\|_{p,\M}<+\infty\big\},
\]
and hope that this collection of test functions is not too small.  

\medskip

\noindent{\sc The bootstrap argument.} 
We proceed with the bootstrap argument. We assume that our parameters are
confined to the intervals $0<p\le1$ and $0\le\theta<+\infty$.
Moreover, we assume that the collection of test functions $\S$ is 
$(p,\theta)$-tame and
that $\| f\|_{p,\M}<+\infty$ holds for each $f\in\S$. We pick  
a normalized element $f\in\S$ with $\| f\|_{p,\M}=1$. 
Since $1=p+(1-p)$, it follows from the fundamental theorem of Calculus that
for $x\le y$,
\begin{equation}
\lvert f(y)-f(x)\rvert=\bigg|\int_{x}^{y}f'(t)\diff t\bigg|\leq 
\int_x^y|f'(t)|\diff t\le \| f'\|_\infty^{1-p}
\int_{x}^y\lvert f'(t)\rvert^p \diff t\le\| f'\|_\infty^{1-p}
\| f'\|_p^p.
\label{eq-basicineq1.1}
\end{equation}
As $f\in L^p(\R)$, the function $f$ must assume values arbitrarily close to 
$0$ on rather big subsets of $\R$. By taking the limit of such $y$ in 
\eqref{eq-basicineq1.1}, we arrive at 
\begin{equation*}
\lvert f(x)\rvert\le\|f'\|_\infty^{1-p}\|f'\|_p^p,\qquad x\in\R,
\end{equation*} 
which gives that
\begin{equation}
\| f\|_\infty\le\|f'\|_\infty^{1-p}\|f'\|_p^p.
\label{eq-basicineq1.3}
\end{equation}
By iteration of the estimate \eqref{eq-basicineq1.3}, we obtain that for 
$n=1,2,3,\ldots$,
\begin{equation}
\|f\|_\infty\le\|f^{(n)}\|_\infty^{(1-p)^n}\|f^{(n)}\|_p^{p(1-p)^{n-1}}\cdots
\|f'\|_p^p=\|f^{(n)}\|_\infty^{(1-p)^n}\prod_{j=1}^{n}\|f^{(j)}\|_p^{(1-p)^{j-1}p}.
\label{eq-basicineq1.4}
\end{equation}
As it is given that $\|f\|_{p,\M}=1$, we have that $\|f^{(j)}\|_p\le M_j$,
which we readily implement into \eqref{eq-basicineq1.4}:
\begin{equation}
\|f\|_\infty\le\|f^{(n)}\|_\infty^{(1-p)^n}\prod_{j=1}^{n}M_j^{(1-p)^{j-1}p} \qquad
\text{if}\quad \|f\|_{p,\M}=1.
\label{eq-basicineq1.5}
\end{equation}
Finally, we let $n$ approach infinity, so that in view of the 
$(p,\theta)$-tameness assumption \eqref{eq-tameness} and homogeneity, 
we obtain that
\begin{equation}
\lVert f\rVert_\infty\le \e^{\theta}\|f\|_{p,\M}\limsup_{n\to+\infty}
\prod_{j=1}^{n}M_j^{(1-p)^{j-1}p},\qquad f\in\S.
\label{eq-basicineq1.6}
\end{equation}
The estimate \eqref{eq-basicineq1.6} tells us that under the requirement 
\begin{equation}
\kappa(p,\M):=\limsup_{n\to+\infty}\sum_{j=1}^{n}(1-p)^{j}\log M_j<+\infty,
\label{eq-condition.nondeg1}
\end{equation}
we may control the sup-norm of a test function $f\in\S$ in terms
of its quasinorm $\|f\|_{p,\M}$. We will refer to the quantity $\kappa(p,\M)$
as the \emph{$p$-characteristic of the sequence $\M$}.
It follows that the Douady-Peetre disconnexion
phenomenon does not occur if we simultaneously control all the higher 
derivatives 
under \eqref{eq-condition.nondeg1} (provided the test function space 
$\S$ is $(p,\theta)$-tame).
But the condition \eqref{eq-condition.nondeg1} achieves more. To see 
this, we first observe that as the inequality \eqref{eq-basicineq1.4} 
applies to an arbitrary smooth function $f$, and in particular to a 
derivative $f^{(k)}$, for $k=0,1,2,\ldots$:
\begin{equation}
\| f^{(k)}\|_\infty\le\|f^{(n+k)}\|_\infty^{(1-p)^n}\prod_{j=1}^{n}
\|f^{(j+k)}\|_p^{(1-p)^{j-1}p},\qquad n=1,2,3,\ldots.
\label{eq-basicineq1.7}
\end{equation}
Next, we let $n$ tend to infinity in \eqref{eq-basicineq1.7} and use 
homogeneity and $(p,\theta)$-tameness as in \eqref{eq-basicineq1.6}, 
and arrive at
\begin{equation}
\|f^{(k)}\|_\infty\le \|f\|_{p,\M}\limsup_{n\to+\infty}
\| f^{(k+n)}\|_\infty^{(1-p)^{n}}\limsup_{n\to+\infty}
\prod_{j=1}^{n}M_{j+k}^{(1-p)^{j-1}p},\qquad f\in\S.
\label{eq-basicineq1.8}
\end{equation}
Moreover, since 
\[
\sum_{j=1}^n(1-p)^j\log M_{j+k}=(1-p)^{-k}\sum_{j=1}^{n+k}(1-p)^j\log M_j
-(1-p)^{-k}\sum_{j=1}^{k}(1-p)^j\log M_j,
\]
and
\[
(1-p)^n\log \| f^{(k+n)}\|_\infty=(1-p)^{-k}(1-p)^{n+k}\log \|f^{(k+n)}\|_\infty
\]
it follows from \eqref{eq-basicineq1.8} and \eqref{eq-condition.nondeg1} 
that 
\begin{equation}
\|f^{(k)}\|_\infty\le\|f\|_{p,\M}\,\e^{\theta(1-p)^{-k}+p(1-p)^{-k-1}\kappa(p,\M)}
\prod_{j=1}^{k}M_{j}^{-(1-p)^{j-k-1}p},\qquad f\in\S,\,\,\,k=0,1,2,\ldots.
\label{eq-basicineq1.9}
\end{equation}
It is immediate from \eqref{eq-basicineq1.9} that under the summability
condition \eqref{eq-condition.nondeg1}, we may in fact control the 
sup-norm of all
the higher order derivatives, which guarantees that the elements of the 
completion of the test function class $\S$ under the quasinorm
$\|\cdot\|_{p,\M}$ consists of $C^\infty$ functions, and the failure of the 
Douady-Peetre disconnexion phenomenon is complete.
We refer to the argument leading up to \eqref{eq-basicineq1.9} as a
``bootstrap'' because we were able to rid ourselves of the sup-norm control
on the right-hand side by diminishing its contribution in the preceding
estimate and taking the limit. 
\medskip

\begin{rem}
The above argument is inspired by an argument which goes back to work 
of Hardy and Littlewood on Hardy spaces of harmonic functions. The phenomenon 
is coined {\em Hardy-Littlewood ellipticity} in \cite{HedenmalmBorichev}. 
To explain the background, we recall that for a function $u(z)$ 
harmonic in the unit 
disk $\mathbb{D}$, the function $z\mapsto \lvert u(z)\rvert^p$ is 
subharmonic if 
$p\geq1$. As such, it enjoys the mean value estimate
\[
\lvert u(0)\rvert^p\leq \frac{1}{\pi}\int_\mathbb{D}\lvert u(z)\rvert^p
\diff A(z).
\]
A remarkable fact is that this inequality survives even for $0<p<1$  
(with a different 
constant) even though subharmonicity fails. See, e.g., 
\cite[Lemma~4.2]{HedenmalmBorichev}, \cite[Lemma~3.7]{Garnett}, 
and the original
work of Hardy and Littlewood \cite{Hardy1932}. 
The similarity with the bootstrap argument used here is striking if 
we compare with
e.g. \cite[Lemma~3.7]{Garnett}.
\end{rem}

\subsection{The $L^p$-Carleman spaces and classes}
The Carleman class 
$\calC_\M$ associated with the weight sequence $\M$ is 
defined as the linear subspace of $f\in C^\infty(\R)$ for which
\[
\frac{\| f^{(k)}\|_\infty}{M_k} \leq  CA^k,
\]
for some positive constant $A=A_f$ (which may depend on $f$).
The theory for Carleman classes was developed in order to 
understand for which classes of functions the (formal) Taylor series 
at a point uniquely determines the function. Denjoy \cite{Denjoy} provided 
an answer under regularity assumptions on the weight sequence, and 
Carleman \cite{Carleman, Carleman1} proved what has since become 
known as the Denjoy-Carleman Theorem: if $\mathcal{N}$ denotes the largest 
logarithmically convex minorant of $\M$, then $\mathcal{C}_\M$ has this 
uniqueness property if and only if
\[
\sum_{j\geq 0}\frac{N_j}{N_{j+1}}=+\infty.
\]
Th{\o}ger Bang later found a simplified proof of this result \cite{Bang2}.
Bang also found numerous other remarkable results for the Carleman classes.
To name one, he proved that the Bang degree $\mathfrak{n}_f$ of 
$f\in\mathcal{C}_\M([0,1])$, defined as the maximal integer $N\ge0$ such that
\[
\sum_{\log \lVert f\rVert_\infty^{-1}<j\leq N }\frac{M_{j-1}}{M_j}<\e,
\]
is an upper bound for the number of zeros of $f$ on the interval $[0,1]$.
As this result was contained in Bang's thesis \cite{Bangthesis}, written in 
Danish, the result appears not to be known to a wider audience.
%This might well be a reason why his results are not widely known.
For an account of several of the interesting results in \cite{Bangthesis}, 
as well as of some further developments in the theory of quasianalytic 
functions, we refer to the work of Borichev, Nazarov, Sodin, and Volberg 
\cite{Sodin1, Sodin2}.

The present work is devoted to the study of the analogous classes defined in 
terms of $L^p$-norms, mainly for $0<p<1$. In view of the preceding subsection, 
it is natural to select the biggest possible collection of test functions 
so that the bootstrap argument has a chance to apply under the assumption 
\eqref{eq-condition.nondeg1}:
\[
\S_{p,\theta,\M}^\circledast:=
\left\{f\in C^{\infty}(\R): \,\,
\lVert f\rVert_{p,\M}<+\infty\,\,\text{ and } 
\limsup_{n\to+\infty}(1-p)^{n}\log\| f^{(n)}\|_\infty\leq\theta
\right\}.
\]
Here, we keep our standing assumptions that $0<p<1$ and $\theta\ge0$ (many
assertion will hold also at the endpoint value $p=1$).
Then $\S_{p,\theta,\M}^\circledast$ automatically meets the asymptotic 
growth condition \eqref{eq-tameness}, but for $(p,\theta)$-tameness to hold we 
also need for each individual derivative to be bounded. However, the condition
\[
\limsup_{n\to+\infty}(1-p)^{n}\log\| f^{(n)}\|_\infty\leq\theta
\]
guarantees that that $f^{(n)}\in L^\infty$ at least for big positive integers
$n$, say for $n\ge n_0$. But then, since $\|f\|_{p,\M}<+\infty$ and in 
particular $f^{(k)}\in L^p$ for all $k=0,1,2,\ldots$, the estimate 
\eqref{eq-basicineq1.3} gives that $f^{(n-1)}\in L^\infty$ as well. Proceeding
iteratively, we find that $f^{(n)}\in L^\infty$ for all $n=0,1,2,\ldots$ if
$f\in\S_{p,\theta,\M}^\circledast$. This means that all the estimates 
of the
preceding subsection are sound for $f\in\S_{p,\theta,\M}^\circledast$. 
Of course, for very degenerate weight sequences $\M$, it might unfortunately 
happen that $\S_{p,\theta,\M}^\circledast=\{0\}$. 
We proceed to define the $L^p$-Carleman spaces.
 
\begin{defn}
$(0<p<1)$ The $L^p$-Carleman space $W^{p,\theta}_\M$ is the completion of the 
test function class $\S_{p,\theta,\M}^\circledast$ with respect to the 
quasinorm $\|\cdot\|_{p,\M}$.
\end{defn}

In the standard textbook presentations, the Carleman classes are defined for 
a regular weight sequence $\M$ in the same way, only the $L^p$ quasinorm 
is replaced by the $L^\infty$ norm, and the class is to be minimal given the 
following two requirements: 
(i)  the space is contained in the class, and (ii) the class is 
invariant with respect to dilatation \cite{Carleman, HORM1, Katznelson}. 
It is easy to see that claiming that $f\in\mathcal{C}_\M$ is the same as 
saying that $f(t)=g(at)$ for some positive real $a$, where 
\[
\frac{\|g^{(k)}\|_\infty}{M_k} \leq C,\qquad k=0,1,2,\ldots,
\]
for some constant $C$, which we understand as the requirement 
$\|g\|_{\infty,\M}<+\infty$ (with exponent $p=+\infty$). 
This allows us to extend the notion of the Carleman classes for exponents
$0<p<1$ as follows. 
 
\begin{defn}
$(0<p<1)$ The $L^p$-Carleman class $\mathcal{C}^{p,\theta}_\M$ consists of 
all the dilates $f_a(t):=f(at)$ (with $a>0$) of functions 
$f\in W^{p,\theta}_\M$. 
\end{defn}

Here, to avoid unnecessary abstraction, we need to understand that
each element in $W_\M^{p,\theta}$ gives rise to an element of the 
Cartesian product space
$L^p\times L^p\times\cdots$ via the lift of the map 
$f\mapsto (f,f',f'',\ldots)$ initially defined on test functions 
$f\in\S_{p,\M}^\circledast$ (for more details, see the next subsection). 
Moreover, it is easy to see that $f\in W^{p,\theta}_\M$ is uniquely 
determined by the corresponding element in $L^p\times L^p\times\ldots$. 
It is important to note that in the Cartesian product space, the 
dilatation operation is well-defined, so that the 
$L^p$-Carleman class $\mathcal{C}^{p,\theta}_\M$ can be understood as a 
submanifold of $L^p\times L^p\times\cdots$ in the above sense.
Moreover, $\mathcal{C}^{p,\theta}_\M$ is actually a linear subspace, as the 
quasinorm criterion for being in the class is (analogously) 
\[
\frac{\| f^{(k)}\|_p}{M_k} \leq C A^k,
\]
for some positive constants $C$ and $A$, and this kind of bound is closed
under linear combination. 
 
\subsection{Classes of weight sequences}
From this point onward, we will restrict attention to positive, logarithmically 
convex weight sequences, i.e. sequences $\M=\{M_j\}_j$ of positive numbers
such that the function $j\mapsto \log M_j$ is convex. 

We will consider the following notions of regularity for weight sequences.

\begin{defn}\label{def:reg-1} A logarithmically convex sequence 
$\M=\{M_n\}_{n=0}^{+\infty}$ 
with infinite $p$-characteristic $\kappa(p,\M)=+\infty$ is called 
\emph{$p$-regular} if one of the following conditions (i)--(ii) holds:

\noindent (i) $\liminf_{n\to{+\infty}}(1-p)^n\log M_n>0$, or

\noindent (ii) $\log M_n=\mathrm{o}((1-p)^{-n})$ and $n\log n\leq 
(\delta+\mathrm{o}(1))\log M_n$ as $n\to{+\infty}$, for some $\delta<1$. 
\label{defn-1.4.3}
\end{defn}
It is a simple observation that $p$-regular sequences are stable 
under the process of shifts and under 
replacing $M_n$ for a finite number of indices, as long as 
log-convexity is kept.

\begin{rem}\label{rem:001}
(a) We note in the context of Definition \ref{defn-1.4.3} that if $\M$ grows 
so fast that $\kappa(p,\M)=+\infty$ holds, then in particular 
the asymptotic estimate $\log M_n=\mathrm{O}(q^{-n})$ fails as $n\to{+\infty}$ 
for any given $q$ with $1-p<q<1$. 
The second condition in (ii), which says that $n\log n\leq 
(\delta+\mathrm{o}(1))\log M_n$, 
is a rather mild lower bound on $\log M_n$  compared with this exponential
growth along subsequences. 

(b) For logarithmically convex sequences $\M$, the sum
\begin{equation}\label{eq:kappa-convex}
\kappa(p,\M)=\sum_{j=0}^{+\infty}(1-p)^j\log M_j
\end{equation}
is actually convergent to an extended real number in $\R\cup\{+\infty\}$.
\end{rem}

We conclude with a notion of regularity that applies in the regime with 
finite $p$-characteristic.

\begin{defn}\label{def:reg-2}
A logarithmically convex weight sequence $\M$ for which 
$\kappa(p,\M)<+\infty$ is said to be \emph{decay-regular} 
if $M_n/M_{n+1}\ge\epsilon^n$ holds for some positive real $\epsilon$, or 
alternatively, if $\M$ meets the nonquasianalyticity condition
\[
\sum_{n=1}^{+\infty}\frac{M_n}{M_{n+1}}<+\infty.
\]
\end{defn}

\subsection{The three phases} 
We mentioned already the phenomenon that the Carleman classes exhibit the 
phase transition associated with quasianalyticity. Here, the concept of 
quasianalyticity is usually defined in terms of the unique continuation 
property that the (formal) Taylor series at any given point determines the 
function uniquely. Under the regularity condition that the sequence 
$\M=\{M_n\}_n$ is logarithmically convex, it is known classically that 
the \emph{Carleman class $\mathcal{C}_\M=\mathcal{C}_{\M}^{\infty,0}$ is 
quasianalytic if and only if}
\[
\sum_{n=0}^{+\infty}\frac{M_n}{M_{n+1}}=+\infty.
\]
In the small exponent range $0<p<1$ considered here, it turns out that we 
have actually \emph{two phase transitions}: 

\noindent (i) the Douady-Peetre disconnexion barrier, and 

\noindent (ii) the quasianalyticity barrier. 

\medskip

Here, we shall attempt to explore both phenomena.

\medskip

In the degenerate case when $M_n=1$ for $n=0,\ldots,k$ and $M_n=+\infty$ for
$n>k$, the Carleman space $W^{p,\theta}_\M$ does not depend on the parameter 
$\theta\ge0$, and is the same as the Sobolev space $W^{p,k}$,  except that 
it is equipped with another (but equivalent) quasinorm. 
So in this instance, we get the Douady-Peetre disconnexion phenomenon 
\eqref{isom} for $W^{p,\theta}_\M$, while under the bounded $p$-characteristic 
condition \eqref{eq-condition.nondeg1}, the following result shows that it 
does not happen. 

To properly formulate the result, we consider the canonical mapping
\mbox{$\bp:W^{p,\theta}_\M\to L^p\times L^p\times\ldots$}, defined initially on 
the test function space $\S_{p,\theta,\M}^\circledast$
by
\[\label{eq-map-canonical}
\bp f=(f, f', f'', \ldots).
\]
It is natural to replace here the product space $L^p\times L^p\times\ldots$ by
its linear subspace $\ell^\infty(L^p,\M)$ supplied with the standard quasinorm:
\begin{equation}
\|(f_0,f_1,f_2,\ldots)\|_{\ell^\infty(L^p,\M)}:=\sup_{n\ge0}\frac{\|f_n\|_p}{M_n}.
\label{eq-quasinormlinfty}
\end{equation}
Indeed, the linear mapping $\bp:W^{p,\theta}_\M\to \ell^\infty(L^p,\M)$ then 
becomes an isometry. This is obvious for test functions in 
$\S_{p,\theta,\M}^\circledast$, and, then automatically holds for elements 
of the abstract completion as well.
We denote the $n$-th component projection of $\bp$ by $\bp_n$: 
$\bp_n(f_0,f_1,f_2,\ldots):=f_n$.

\begin{thm}\label{thm-smoothmap}
$(0<p<1)$ Suppose that the weight sequence $\M$ is logarithmically 
convex and meets the finite $p$-characteristic condition 
\eqref{eq-condition.nondeg1}. 
%in the strong form
%that the sum 
%\[
%\kappa(p,\M)=\sum_{j=1}^{+\infty}(1-p)^j\log M_j
%\] 
%converges. 
Then,  for each $n=0,1,2\ldots$,
$\bp_n$ maps $W_\M^{p,\theta}$ into $C^\infty(\R)$, and the space $W_\M^{p,\theta}$ 
is coupled, in the sense that
\[
\partial\bp_n f=\bp_{n+1}f,\quad f\in W_\M^{p,\theta},
\]
where $\partial$ stands for the differentiation operation.
Moreover, the projection $\bp_0$ is injective, and, in the natural sense, 
the space equals the collection of test functions: 
$W_\M^{p,\theta}=\S_{p,\theta,\M}^\circledast$.
\end{thm}

The proof of this theorem is supplied in Subsection~\ref{smooth_regime}.

The conclusion of Theorem~\ref{thm-smoothmap} is that under the the strong 
finite $p$-characteristic condition, $W_\M^{p,\theta}$ is a space of smooth 
functions, and, indeed, it is identical with the test function class 
$\S_{p,\theta,\M}^\circledast$.
The situation is drastically different when $\kappa(p,\M)=+\infty$. 

\begin{thm}\label{thm-uncoupling}
$(0<p<1)$
Suppose $\M$ is a $p$-regular sequence with infinite $p$-characteristic 
$\kappa(p,\M)=+\infty$. 
Then $\bp_0$ is surjective onto $L^p(\R)$, and $\bp$ supplies an isomorphism 
\[
W^{p,\theta}_\M\cong L^p(\R)\oplus W^{p,\theta_1}_{\M_1},
\]
where $\M_1$ denotes the shifted sequence 
$\M_1:=\{M_{n+1}\}_n$ and $\theta_1=\theta/(1-p)$. Moreover, 
$\M_1$ inherits the assumed properties of $\M$.
\end{thm}

A sketch of the proof of Theorem \ref{thm-uncoupling} is supplied in 
Subsection \ref{subsec-bbbg}. As for the formulation, the summand 
$W^{p,\theta_1}_{\M_1}$ on the right-hand side arises as the space
of ``derivatives'' of functions in $W^{p,\theta}_\M$. Here, the reason why 
$\M$ gets replaced by $\M_1$ is due to a one-unit shift in the sequence 
space $\ell^\infty(L^p,\M)$.  Moreover, the reason why $\theta$ gets replaced by
$\theta_1=\theta/(1-p)$ is the corresponding shift in the space of test 
functions $\S_{p,\theta,\M}$ when we take the derivative.

\begin{rem}
Since $\M_1$ inherits all relevant properties from $\M$, the theorem 
can be applied iteratively to obtain an isomorphism
\[
W_\M^{p,\theta}\cong L^p\oplus L^p\oplus\ldots \oplus 
L^p\oplus W_{\M_k}^{p,\theta_k},\quad k\in\mathbb{N},
\]
where $\theta_k=\theta(1-p)^{-k}$. 
\end{rem}

We briefly comment on the remaining transition, between non-quasianalyticity 
and quasianalyticity.
For $\theta=0$, we characterize quasianalyticity for 
$\mathcal{C}_\M^{p,\theta}=\mathcal{C}_\M^{p,0}$ in terms of the quasianalyticity
of the standard Carleman class $\mathcal{C}_{\wtM}$ for a certain 
related sequence $\wtM$. 
Moreover, the classical Denjoy-Carleman theorem supplies criteria for when 
the class $\mathcal{C}_{\wtM}$ is quasianalytic or non-quasianalytic. 
The associated sequence $\wtM=\{\twtM_n\}_n$ is given by
\[
\twtM_n:=\prod_{j=1}^\infty M_{n+j}^{(1-p)^{j-1}p},\quad k=0,1,2,\ldots,
\]
which we recognize as coming from the $L^\infty$-bound of the higher order 
derivatives in \eqref{eq-basicineq1.9}.  

The result runs as follows.

\begin{thm}\label{thm-quasianalytic}
$(0<p<1)$
Assume that $\M$ is logarithmically convex and bounded away from zero. 
If $\kappa(p,\M)<+\infty$, the following holds:

\noindent{\rm(i)} If $\theta>0$, then $\mathcal{C}_\M^{p,\theta}$ is never
quasianalytic.

\noindent{\rm(ii)} 
If $\theta=0$ and $\M$ is decay-regular in the sense of 
Definition~\ref{def:reg-2},
then $\mathcal{C}_\M^{p,0}$ is quasianalytic if and only if 
$\mathcal{C}_{\wtM}$ is quasianalytic.
\end{thm}

Finally, we comment on the dependence of the 
classes on the parameter $\theta$ in the smooth context of Theorem 
\ref{thm-smoothmap}.

\begin{thm}\label{theta-thm}$(0<p<1)$
Let $\M$ be an increasing, log-convex sequence such that 
$\kappa(p,\M)<+\infty$. Let $0\leq \theta<\theta'$. Then the inclusion 
$W_\M^{p,\theta}\subset W_\M^{p,\theta'}$ is strict: 
$W_\M^{p,\theta}\ne W_\M^{p,\theta'}$.
\end{thm}

We do not know whether such a strict inclusion holds in the 
uncoupled regime when $\kappa(p,\M)=+\infty$. 
It remains a possibility that the spaces are then so large that the 
parameter $\theta$ is not felt. In any case, we are able to show that
(Proposition \ref{prop:inclusion-tail-decay}) 
\begin{equation}
c_0(L^p,\M)\subset\bp W_\M^{p,\theta}\subset \ell^\infty(L^p,\M).
\label{eq-containment1}
\end{equation}
Here, $c_0(L^p,\M)$ denotes the subspace of $\ell^\infty(L^p,\M)$
consisting of sequences $(f_0,f_1,f_2,\ldots)$ with
\[
\lim_{n\to+\infty}\frac{\|f_n\|_p}{M_n}=0. 
\]

\begin{rem}
(a) In view of the Douady-Peetre disconnexion phenomenon, the fact that for
$0<p<1$, $L^p$ functions fail to define distributions is a serious
obstruction. An alternative approach is to consider the real Hardy spaces 
$H^p$ in place of $L^p$, since $H^p$ functions automatically define 
distributions. The drawback of that approach is that for $p=1$, $H^1$ is 
substantially smaller than $L^1$. Our theme here is to keep $L^p$ and to let
a bootstrap argument (involving infinitely many higher order derivatives) 
take care of the smoothness, and to explore what happens
when the bootstrap argument fails to supply appropriate bounds.

(b) A word on the title. The term {\em a critical topology} used here is 
borrowed from Beurling's work \cite{Beuart1}, where another phase transition 
is the object of study. 
\end{rem}

\section{Sobolev spaces: Peetre's proof and failure 
of embedding}\label{sec-sobolev}

\subsection{Sobolev spaces for $0<p<1$}
We fix a number $p$ with $0<p<1$ and an 
integer $k\geq0$. 
Following Peetre \cite{Peetre} we consider the Sobolev space 
$W^{k,p}=W^{k,p}(\R)$, 
defined as the abstract completion of $C^k_0(\R)$
with respect to the quasinorm
\begin{equation}
\lVert f\rVert_{k,p}=\left(\lVert f \rVert_p^p+\lVert f' \rVert_p^p+\ldots 
+ \lVert f^{(k)} \rVert_p^p\right)^{1/p}.
\label{eq-sobolevnorm2}
\end{equation}
The resulting space $W^{k,p}$ is then a quasi-Banach space. Here, 
$C^k_0(\R)$ denotes the space of compactly supported functions in $C^k(\R)$.

\begin{rem} 
In this paper we shall mostly work on the entire line. If at some place 
we consider spaces on bounded intervals, this will be explicitly mentioned. 
The definition of $W^{k,p}(I)$ for a general interval $I$ is entirely analogous.
\end{rem}

The space $W^{k,p}$ comes with two canonical mappings, 
$\ba=\ba_k:W^{k,p}\to L^p$, and $\bd=\bd_k: W^{k,p}\to W^{k-1,p}$. 
These are both initially defined for test functions $f\in C^{k}_0(\R)$ by
\[
\ba f=f,\quad \text{and}\quad \bd f=f'.
\]
The mappings $\ba$ and $\bd$ are bounded and densely defined, and hence
extend to bounded operators on the entire space $W^{k,p}$.

\subsection{Douady-Peetre disconnexion}
We begin this section by considering a simple example which explains how a 
crucial feature differs in the setting of $0<p<1$ as compared to the 
classical Sobolev space case. 

We are used to thinking of $W^{k,p}$ as being a certain subspace of $L^p$, 
consisting of functions that are sufficiently smooth. 
As mentioned in the introduction, the first observation, made by Douady, 
is that this is not the right way to think when $0<p<1$. Indeed, the 
canonical map $\ba$ is not injective.

\begin{prop}[Douady]
There exists $f\in W^{1,p}([0,1])$ such that
\[
\ba f=0,\quad \bd f=1.
\]
\end{prop}

This would suggest that there ought to exist functions that vanish 
identically but nevertheless the derivative equals the nonzero constant $1$.
This is of course absurd, and the right way to think about it is to realize
that in the completion, the function and its derivative lose contact, 
they disconnect.

\begin{proof} 
A small argument (see \cite[Lemma~2.1]{Peetre}) shows that  we are allowed 
to work with functions whose derivatives have jumps. 
We let $\{\epsilon_j\}_j$ be a sequence of positive reals, such that 
$j \epsilon_j\to0$ as $j\to+\infty$, and define $f_j$ on the interval 
$[0,\frac{1}{j}+\epsilon_j]$ by
\[
f_j(x)=\begin{cases}x, & 0\leq x \leq \frac{1}{j}, \\ 
(\frac{1}{j}+\epsilon_j-x)/(j\epsilon_j), & 
\frac{1}{j}<x\leq\frac{1}{j}+\epsilon_j,
\end{cases}
\]
and extend it \emph{periodically} to $\R$ with period $\frac{1}{j}+\epsilon$. 
The resulting function $f_j$ will be a skewed saw-tooth 
function that rises slowly with slope $1$ and then drops steeply. 
By differentiating $f_j$, we have that
\[
1-f_j'(x)=
\begin{cases} 0, & 0\leq x<\frac{1}{j},
\\ 
1+\frac{1}{j \epsilon_j},& 
\frac{1}{j}<x< \frac{1}{j}+\epsilon_j.
\end{cases}
\]
Since $f_j$ assumes values between $0$ and $\frac{1}{j}$, it is clear 
that $f_j\to 0$ as $j\to+\infty$ in $L^\infty$ and hence in $L^p$. 
Within the interval $[0,1]$, there are at most $j$ full periods of the
function $f_j$,
which allows us to estimate
\[
\int_{[0,1]}|1-f_j'(x)|^p\diff x\le 
j\epsilon_j\bigg(1+\frac{1}{j\epsilon_j}\bigg)^p\le 
j\epsilon_j(j\epsilon_j)^{-p}
=(j\epsilon_j)^{1-p}\to0\quad\text{as}\,\,\,\,j\to+\infty.
\]
In view of the above observations, $f_j\to0$ while $f_j'\to1$, both 
in the quasinorm of $L^p$, as $j\to+\infty$. 
In particular, $\{f_j\}_j$ is a Cauchy sequence, and if we let $f$ 
denote the abstract limit in the completion, we find that 
$\ba f=0$ while $\bd f=1$.
\end{proof}

\subsection{The isomorphism and construction of the 
canonical lifts}
\label{subsec-isomconstr}
We fix an integer $k=1,2,3,\ldots$ and an exponent $0<p<1$. 

The space $L^p\oplus W^{k-1,p}$ consists of pairs $(g,h)$, where $g\in L^p$ and
$h\in W^{k-1,p}$, and we equip it with the quasinorm
\[
\lVert (g,h)\rVert^p=\lVert g\rVert^p_p+\lVert h\rVert_{k-1,p}^p,
\qquad g\in L^p,\,\,\,h\in W^{k-1,p}.
\]
From the definition of the norm \eqref{eq-sobolevnorm2}, we see that 
the operator $\Aop:\,W^{k,p}\to L^p\oplus W^{k-1,p}$ given by 
$\Aop f:=(\ba f,\bd f)$ is an \emph{isometry}. Indeed, for $f\in C^k_0(I)$, we 
have that
\begin{equation}
\|\Aop f\|^p=\|(\ba f,\bd f)\|^p=\|\ba f\|^p_p+\|\bd f\|^p_{k-1,p}=
\|f\|_p^p+\|f'\|^p_{k-1,p}=\|f\|_{k,p}^p,
\label{eq-isom1.1}
\end{equation}
and this property survives the completion process. If $\Aop$ can be shown to
be surjective, then it is an isometric isomorphism 
$\Aop:\,W^{k,p}\to L^p\oplus W^{k-1,p}$. Proceeding iteratively with $W^{k-1,p}$,
we obtain the desired decomposition, since clearly $W^{0,p}=L^p$.

To obtain the surjectivity of $\Aop$, we shall construct two canonical lifts, 
$\bb:\, L^p\to W^{k,p}$ and 
$\bg:\,W^{k-1,p}\to W^{k,p}$ of $\ba$ and $\bd$, respectively. 
These are injective mappings, from $L^p$ and $W^{k-1,p}$ to $W^{k,p}$, 
respectively, satisfying certain relations with $\ba$ and $\bd$. The 
properties of
these are summarized in the following lemma (the notation $\id_X$ stands 
for the 
identity mapping on the space $X$). The details of the construction are 
postponed 
until Section~\ref{ss:construct}.

\begin{lem}
\label{lifts} 
For each $k=1,2,3,\ldots$, there exist bounded linear mappings 
$\bb: L^p\to W^{k,p}$ and $\bg: W^{k-1,p}\to W^{k,p}$, such that

\begin{equation*}
\ba\bb= 
\id_{L^p},\quad
\bd\bg=\id_{W^{k-1,p}},\quad
\bd\bb=
0,\quad 
\ba\bg=0. 
\end{equation*}
\end{lem}

With this result at hand, the proof of the main theorem about the 
$W^{k,p}$-spaces  becomes a simple exercise.

\begin{proof}[Proof of Theorem~\ref{Peetre}]
As noted above, it will be enough to show that the isometry 
$\Aop:\,W^{k,p}\to L^p\oplus W^{k-1,p}$ given by $\Aop f:=(\ba f,\bd f)$ 
is surjective.
To this end, we pick $(g,h)\in L^p\oplus W^{k-1,p}$. 
Then $\bb g$ and $\bg h$ are both elements of $W^{k,p}$, and so is their sum
\[
f=\bb g+\bg h\in W^{k,p}.
\]
It now follows from Lemma \ref{lifts} 
that
\[
\Aop f=(\ba(\bb g + \bg h), \bd(\bb g + \bg h))=(\ba\bb g+\ba\bg h, 
\bd\bb g+\bd\bg h)=(g,h).
\]
As a consequence, $\Aop$ is surjective, and hence $\Aop$ induces an isometric
isomorphism
\[
W^{k,p}\cong L^p\oplus W^{k-1,p}.
\]
By iteration of the same argument, the claimed decomposition of $W^{k,p}$ 
follows.
\end{proof}
\begin{rem}
The lift $\bg$ does not appear in Peetre's work \cite{Peetre}.
Reading between the lines one can discern its role, but here we fill in the 
blanks and treat it explicitly.
\end{rem}

\section{Construction of lifts, and invisible mollifiers}
\subsection{A collection of smooth functions by iterated 
convolution}
\label{ss:convol}
For an integer $k\ge0$ and 
a real $0\le\alpha\le1$ let $C^{k,\alpha}$ denote the class of $k$ times 
continuously
differentiable functions, whose derivative of order $k$ is H\"older continuous
with exponent $\alpha$. 
Given two functions $f,g\in L^1(\R)$, their convolution $f*g\in L^1(\R)$ is as
usual given by
\[
(f*g)(x)=\int_\R f(x-t)g(t) \diff t,\quad x\in\R.
\]
For $a>0$, we let the function $H_a$ denote the normalized characteristic 
function
$H_a=a^{-1}1_{[0,a]}$. For a decreasing sequence of positive reals 
$a_1,a_2,a_3,\ldots$, 
consider the associated repeated convolutions (for $j\le k$)
\begin{equation}\label{eq:convol-def}
\Phi_{j,k}:=H_{a_j}*\cdots*H_{a_k}.
\end{equation}
The function $\Phi_{j,k}$ then has compact support $[0,a_j+\cdots a_k]$ and
belongs to the smoothness class $C^{k-j,1}$ which means that the derivative 
of order $k-j$ is Lipschitz continuous.
We will assume that the sequence $a_1,a_2,a_3,\ldots$ decreases to $0$ at 
least fast enough for $(a_j)_{j\geq 1}\in \ell^1$ to hold. 
%At other times, it will be convenient 
%to require that the tail-sum bound
%\[
%\sum_{j\geq n}a_j\leq C a_n,\quad n\geq 0,
%\] 
%holds for some constant $C$ with $0<C<+\infty$.
Then we may form the limits
\[
\Phi_{j,\infty}:=\lim_{k\to+\infty} \Phi_{j,k},\qquad j=1,2,3,\ldots,
\]
and see that each such limit $\Phi_{j,\infty}$ is $C^\infty$-smooth with support
$[0,a_j+a_{j+1}+\dots]$. Moreover, we have the sup-norm controls
\begin{equation}
\|\Phi_{j,k}\|_\infty\le \frac{1}{a_j},\qquad 
\|\Phi_{j,\infty}\|_\infty\le \frac{1}{a_j}.
\label{eq-elest1.001}
\end{equation}
Next, since for $l<k$
\[
\Phi_{j,k}=\Phi_{j,l}*\Phi_{l+1,k} \quad\text{and}\quad
\Phi_{j,\infty}=\Phi_{j,l}*\Phi_{l+1,\infty},
\]
we may calculate the higher order derivatives by the  formula
\[
\Phi_{j,k}^{(n)}=\Phi_{j,j+n-1}^{(n)}*\Phi_{j+n,k} \quad\text{and}\quad
\Phi_{j,\infty}^{(n)}=\Phi_{j,j+n-1}^{(n)}*\Phi_{j+n,\infty},
\]
interpreted when needed in the sense of distribution theory. Here, we should
ask that $j+n\leq k+1$ in the first formula. By calculation,
\[
\Phi_{j,j+n-1}^{(n)}=\frac{1}{a_j\cdots a_{j+n-1}}(\delta_{a_j}-\delta_0)*\cdots*
(\delta_{a_{j+n-1}}-\delta_0),
\]
which when expanded out is the sum of delta masses at $2^n$ 
(generically distinct) points, each with mass $(a_j\cdots a_{j+n-1})^{-1}$.   
By the convolution norm inequality $\|f*g\|_\infty\le\|f\|_1\|g\|_\infty$, where
the $L^1$ norm may be extended to the finite Borel measures, we have that
\[
\big\|\Phi_{j,k}^{(n)}\big\|_\infty=\big\|\Phi_{j,j+n-1}^{(n)}\big\|_1
\big\|\Phi_{j+n,k}\big\|_\infty 
\le \frac{2^n}{a_j\cdots a_{j+n}},
\]
where we used the estimate \eqref{eq-elest1.001}. The analogous estimate holds
for $k=\infty$ as well:
\begin{equation}
\big\|\Phi_{j,\infty}^{(n)}\big\|_\infty\le\frac{2^n}{a_j\cdots a_{j+n}}.
\label{eq-elest1.002}
\end{equation}
We need to estimate the $L^p$-norm of the function $\Phi_{j,\infty}^{(n)}$ as 
well. The standard norm estimate for convolutions is 
$\|f*g\|_q\le\|f\|_1\|g\|_q$ which holds provided that $1\le q\le+\infty$.
For our small exponents $0<p<1$ this is no longer true. 
However, there is a substitute, provided $f$ is a finite sum of point masses:
\[
\|f*g\|_p^p\le\|f\|_{\ell^p}^p\|g\|_p^p, \quad\text{where}\quad \|f\|_{\ell^p}^p
=\sum_j |b_j|^p\quad\text{if}\quad f=\sum_j b_j\delta_{x_j},
\] 
for some finite collection of reals $x_j$. This follows immediately from the
$p$-triangle inequality and the translation invariance of the $L^p$-norm.
In our present context we see that
\begin{equation}
\big\|\Phi_{j,k}^{(n)}\big\|_p^p=\big\|\Phi_{j,j+n-1}^{(n)}*\Phi_{j+n,k}
\big\|_p^p\le\big\|\Phi_{j,j+n-1}^{(n)}\big\|_{\ell^p}^p\big\|\Phi_{j+n,k}\big\|_p^p
\le \frac{2^n}{(a_j\cdots a_{j+n})^p}\sum_{l=j+n}^{k}a_l,
\label{eq-elest1.003}
\end{equation}
where $n+j\leq k+1$. Correspondingly for $k=+\infty$ we have that
\begin{equation}
\big\|\Phi_{j,\infty}^{(n)}\big\|_p^p=\big\|\Phi_{j,j+n-1}^{(n)}*\Phi_{j+n,\infty}
\big\|_p^p\le\big\|\Phi_{j,j+n-1}^{(n)}\big\|_{\ell^p}^p
\big\|\Phi_{j+n,\infty}\big\|_p^p
\le \frac{2^n}{(a_j\cdots a_{j+n})^p}\sum_{l=j+n}^{+\infty}a_l.
\label{eq-elest1.004}
\end{equation}

\subsection{Existence of invisible mollifiers in $W^{k,p}$}
\label{ss:construct}
We now employ the repeated convolution procedure of Section~\ref{ss:convol},
to exhibit mollifiers with $L^p$-vanishing properties.
\begin{lem}[invisibility lemma]
\label{lem:mollifier-finite} 
Let $k\geq 0$ be an integer. Then for any given $\epsilon>0$, there exists a 
non-negative function $\Phi \in C^{k,1}(\R)$ such that
\[
\int_\R \Phi\; \diff x=1,\quad \operatorname{supp}(\Phi)\subset 
[0,\epsilon], \quad \lVert \Phi\rVert_{k,p} < \epsilon.
\]
\end{lem}

\begin{proof}
For positive decreasing numbers $\{a_j\}_{j=1}^{k+1}$ (to be determined), 
let $\Phi=\Phi_{1,k+1}$ be given by \eqref{eq:convol-def}.
Then clearly $\Phi\in C^{k,1}(\R)$ with $\Phi\ge0$ and 
$\int_\R \Phi\,\diff x=1$. 
Moreover, the support of $\Phi$ equals $[0,a_1+\cdots+a_{k+1}]$.
By the estimate \eqref{eq-elest1.003}, it follows that
\begin{equation}\label{eq:norm-conv-finite} 
\lVert \Phi^{(n)}\rVert_p^p\leq \frac{2^n}{(a_1\cdots a_{n+1})^p}
\sum_{l=n+1}^{k+1}a_l.
\end{equation}
We need to show that the finite sequence $\{a_l\}_{l=1}^{k+1}$ may be chosen 
such that the sum of the right-hand side in \eqref{eq:norm-conv-finite} 
over $0\leq n\leq k$ is bounded by $\epsilon^p$, while at the same time
$\sum_{l=1}^{k+1} a_l\leq\epsilon$.
As a first step, \emph{we assume that $a_{j+1}\leq \frac12{a_{j}}$ for integers
$n\ge0$}, and observe that it then follows that $\sum_{l=j}^{k+1}a_l\leq 2a_j$. 
Consequently, we get that $\supp\Phi\subset [0,2a_0]$ and
\[
\lVert \Phi^{(n)}\rVert_p^p\leq 2^{n+1}\frac{a_{n+1}}{(a_1\cdots a_{n+1})^p},\quad 
n=0,\ldots,k.
\]
We put
\[
a_1:=\min\left\{\left(\frac{\epsilon^p}{2(k+1)}\right)^{1/(1-p)}, 
\frac{\epsilon}{2}\right\}
\] 
and successively declare that 
\[
a_l:=\min\left\{
\Bigg(\frac{\epsilon^p (a_1\cdots a_{l-1})^p}
{2^{l}(k+1)}\Bigg)^{1/(1-p)},\frac{a_{l-1}}{2}\right\}, \quad l=2,\ldots k+1.
\]
It then follows that
\[
\lVert \Phi^{(n)}\rVert_p^p\leq2^{n+1}\frac{a_{n+1}}{a_1^p\cdots a_{n+1}^p}
=a_{n+1}^{1-p}\frac{2^{n+1}}{(a_0\cdots a_{n})^p}\leq\frac{\epsilon^p}{k+1},
\quad n=0,\ldots,k.
\]
whence $\lVert \Phi\rVert_{p,k}\leq \epsilon$ and since also
$\supp\Phi\subset [0,2a_0]\subset [0,\epsilon]$, the constructed function 
$\Phi$ meets all the specifications.

\end{proof}

\subsection{The definition of the lift $\bb$ for $0<p<1$.}
\label{ss:lift-b}
%\noindent{\sc The definition of the lift $\bb$ for $0<p<1$.} 
The lift $\bb$ 
maps boundedly $L^p\to W^{k,p}$, and we need to explain how it gets to be 
defined. Let $\mathscr{F}$ denote the collection of \emph{step functions},
which we take to be the finite linear combination of characteristic
functions of bounded intervals, and when also equip it with the quasinorm 
of $L^p$, we denote it by $\mathscr{F}_p:=(\mathscr{F},\|\cdot\|_p)$. We 
note that $\mathscr{F}_p$ is quasinorm dense in $L^p$. We first define 
$\bb g$ for $g\in\mathscr{F}_p$. For $g\in\mathscr{F}_p$, we will write down a 
$W^{k,p}$-Cauchy sequence $\{g_{j}\}_j$ of test functions $g_{j}\in C^k_{0}(\R)$, 
and declare $\bb g\in W^{k,p}$ to be the abstract limit of the  
Cauchy sequence $g_{j}$ as $j\to+\infty$. 

We will require the following properties of the test functions $g_j$:
\begin{equation}
\lim_{j\to+\infty}\lVert g-g_{j}\rVert_p=0\quad\text{and}
\quad \lim_{j\to+\infty}\lVert g_{j}'\rVert_{k-1,p}=0.
\label{eq-defbb1.1}
\end{equation}
If \eqref{eq-defbb1.1} can be achieved, then 
$\bb:\,\mathscr{F}_p\to W^{k,p}$ becomes an isometry by the following 
calculation:
\[\label{eq:beta-eqns}
\|\bb g\|_{k,p}^p=\lim_{j\to+\infty}\|g_{j}\|_{k,p}^p
=\lim_{j\to+\infty}
\big(\|g_{j}\|_{p}^p+\|g_{j}'\|_{k-1,p}^p\big)=\|g\|_p^p,\qquad g\in\mathscr{F}.
\]

\emph{These properties uniquely determine $\bb g$ for $g\in\mathscr{F}$}. 
Indeed, if $\tilde g_{j}$ were another 
Cauchy sequence satisfying \eqref{eq:beta-eqns}, then $\{\tilde g_{j}\}_j$ and 
$\{g_{j}\}_j$ are equivalent as Cauchy sequences, in light of 
\[
\lVert \tilde g_{j}-g_{j}\rVert_{k,p}^p= \lVert \tilde g_{j}-g_j\rVert_p^p
+\lVert \tilde g_{j}'-g_{j}'\rVert_{k-1,p}^p\le\lVert \tilde g_{j}-g_j\rVert_p^p
+\| \tilde g_{j}'\|_{k-1,p}^p+\|g_{j}'\|_{k-1,p}^p.
\]
In particular, $\bb:\mathscr{F}_p\to W^{k,p}$ is then a well-defined bounded 
operator, and since $\mathscr{F}_p$ is dense in $L^p$ it extends uniquely 
to a bounded operator $\bb:L^p\to W^{k,p}$ which is actually an isometry. 

In view of the above, it will be enough to define $\bb g$ when $g$ is the 
characteristic function of an interval $g=1_{[a,b]}$ and to check 
\eqref{eq-defbb1.1} for it, since general step functions in $\mathscr{F}$
are obtained using finite linear combinations. 

\medskip

Let $\{\epsilon_j\}_{j=0}^{+\infty}$ be a sequence of numbers tending to zero 
with $0<\epsilon_j<\frac12(b-a)$. By 
Lemma~\ref{lem:mollifier-finite} applied to $W^{k-1,p}$, there exists 
non-negative functions $\Phi_{\epsilon_j}\in C^{k-1}$, such that
\[
\int_\R \Phi_{\epsilon_j}\diff x=1,\quad 
\supp \Phi_{\epsilon_j}\subset [0,\epsilon_j],
\quad \nl \Phi_{\epsilon_j}\nr_{k-1,p}<\epsilon_j.
\]
We define $g_j$ by convolution: $g_j:=\Phi_{\epsilon_j}*1_{[a,b]}$. 
It is then clear that $g_j-g$ has support on 
$[a,a+\epsilon_j]\cup [b,b+\epsilon_j]$, and there, it is bounded by $1$ in 
modulus. As a consequence,
\[
\|g_j-g\|_p^p\leq 2\epsilon_j,
\]
so $g_j\to g$ in $L^p$. Next, we consider the derivative 
$g_j'$, which we may express as $g_j'=(\tau_a-\tau_b)\Phi_{\epsilon_j}$, where
we recall that $\tau$ with subscript is a translation operator. 
It is clear that
\[
\| g_j'\|_{k-1,p}^p\leq 2\| \Phi_{\epsilon_j}\|^p_{k-1,p}\leq 2\epsilon_j.
\]
Consequently, $\|g_j'\|_{k-1,p}^p$ also tends to zero, as needed. This 
establishes \eqref{eq-defbb1.1}.
\medskip

\begin{figure*}[h]
    \centering
    \includegraphics*[width=0.6\textwidth]{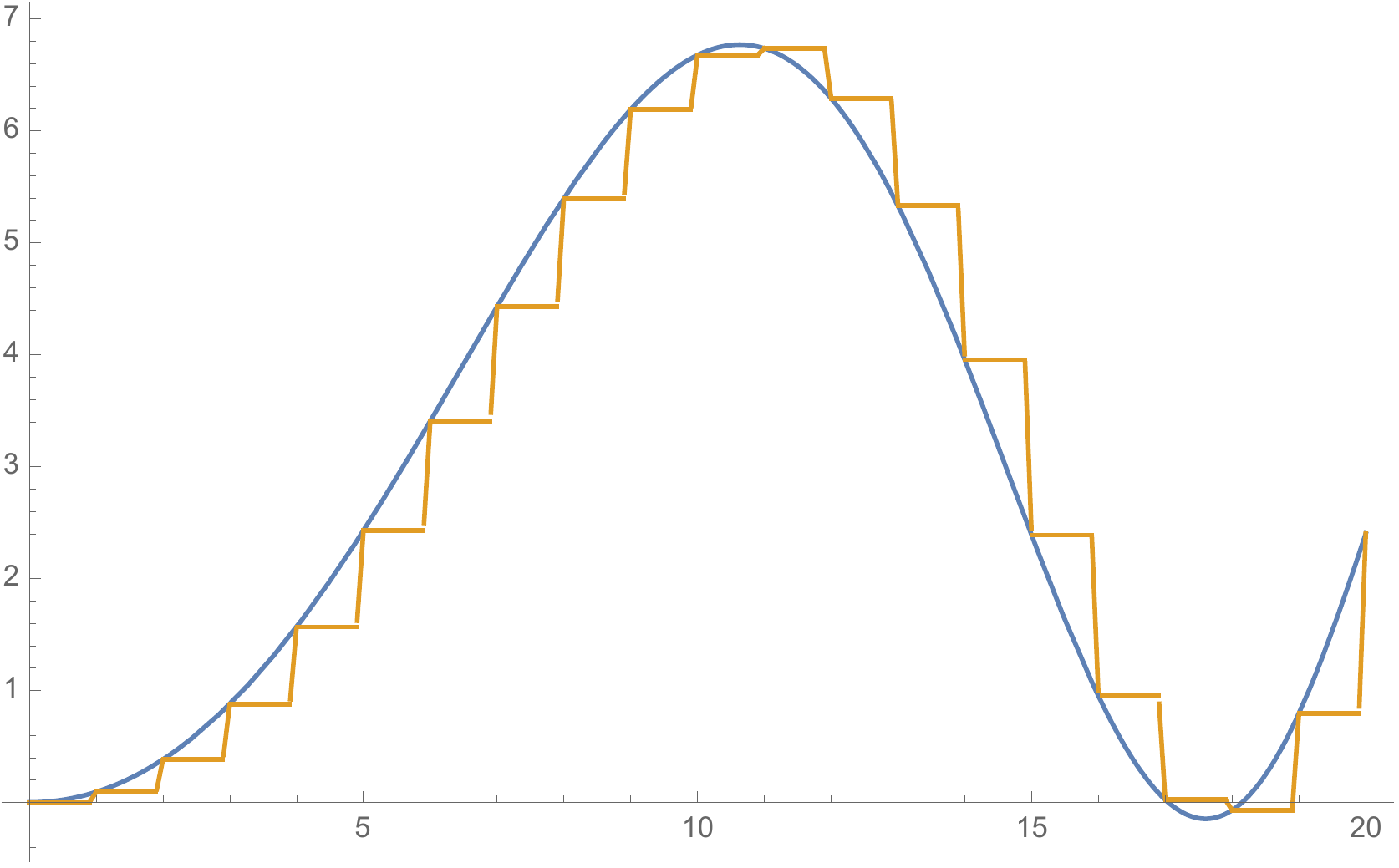}
    \caption{The construction of $\bb$ for $W^{1,p}$.}
    \label{fig:beta}
\end{figure*}

\begin{proof}[Proof of Lemma, part 1]
We show that $\ba\bb=\id_{L^p}$ and $\bd\bb=0$. 
Since $\bb g\in W^{k,p}$ is the abstract limit of the Cauchy sequence 
$g_j$ with \eqref{eq-defbb1.1}, and by definition $\ba g_j=g_j$ and 
$\bd g_j=g_j'$, it follows from \eqref{eq-defbb1.1} that $\ba\bb g=g$ 
and $\bd\bb g=0$ for every $g\in L^p$.
The assertion follows.
\end{proof} 

%\medskip
\subsection{The lift $\bg$.}\label{ss:lift-g}
%\noindent{\sc The lift $\bg$.} 
To construct $\bg$, we let $g\in W^{k-1,p}$
be an arbitrary element, which is by definition the abstract limit of 
some Cauchy sequence $\{g_j\}_j$, where $g_j\in C^{k-1}_0(\R)$. 
For any given $\epsilon>0$, Lemma~\ref{lem:mollifier-finite} provides a 
function $\Phi_\epsilon\in C^{k-1}_0(\R)$
with $\Phi_\epsilon\ge0$, $\langle \Phi_\epsilon\rangle_\R
:=\int_\R \Phi_\epsilon(t)\diff t=1$, supported in $[0,\epsilon]$, 
while at the same time, $\nl\Phi_\epsilon\nr_{k-1,p}<\epsilon$.  
We use the functions $\Phi_{\epsilon_j}$ to modify each $g_j(x)$ to 
have vanishing zeroth moment, by defining
\[
\tilde{g}_j(x):=g_j(x)-\langle g_j\rangle_\R \Phi_{\epsilon_j}(x),\qquad
\langle g_j\rangle_\R:=\int_\R g_j(t)\diff t,
\]
where the $\epsilon_j$ are chosen to tend to zero so quickly that
\[
\lim_{j\to+\infty} \| g_j-\tilde{g}_j\|_{k-1,p}=
\lim_{j\to+\infty}\|\Phi_{\epsilon_j}\|_{k-1,p}\lvert\langle g_j\rangle_\R\rvert=0.
\]
Next, we define the functions $u_j$ as primitives:
\[
u_j(x)=\int_{-\infty}^x\tilde{g}_j(t)\diff t,\quad x\in\R.
\]
Then as $\tilde g_j$ has integral $0$, we see that $u_j\in C^k_0(\R)$.
We put $f_j:= u_j-\bb u_j\in W^{k,p}$, and observe that from the known 
properties of $\bb$, it follows that
\begin{equation}
\ba f_j=\ba u_j-\ba\bb u_j=u_j-u_j=0\quad\text{and}\quad
\bd f_j=\bd u_j-\bd\bb u_j=\bd u_j=u_j'=\tilde g_j.
\label{eq-propabgd}
\end{equation}
Then, from the isometry of $\Aop:\,W^{k,p}\to L^p\oplus W^{k-1,p}$
(see \eqref{eq-isom1.1}), we have that 
\begin{equation}
\|f_j\|_{k,p}^p=\|\ba f_j\|_p^p+\|\bd f_j\|_{k-1,p}^p
=\|\tilde g_j\|^p_{k-1,p}\le\| g_j\|^p_{k-1,p}+
\|\Phi_{\epsilon_j}\|_{k-1,p}^p|\langle g_j\rangle_\R|^p,
\label{eq-contractive1.23}
\end{equation}
where in the last step, we applied the $p$-triangle inequality. 
A similar verification shows that $\{f_j\}_j$ is a Cauchy sequence, so that
it has a limit $\bg g:=\lim_{j\to+\infty}f_j$ in $W^{k,p}$. Moreover, in view of 
\eqref{eq-propabgd}, it follows that
\begin{equation}
\ba\bg g=\lim_{j\to+\infty}\ba f_j=0\quad\text{and}\quad
\bd\bg g= \lim_{j\to+\infty} \bd f_j=\lim_{j\to+\infty}\tilde g_j=g,
\label{eq-propabgd2.01}
\end{equation}
in $L^p$ and $W^{k-1,p}$, respectively. In the construction of the sequence of
functions $f_j$, there is some arbitrariness e.g. in the choice of the sequence
of the $\epsilon_j$ (they were just asked to tend to $0$ sufficiently 
quickly).  
To investigate whether this matters, we suppose another Cauchy sequence 
$\{F_j\}_j$ in $W^{k,p}$ is given, with properties that mimic 
\eqref{eq-propabgd}: that $\ba F_j=0$ in $L^p$, and that for some 
Cauchy sequence $\{G_j\}_j$ in $W^{k-1,p}$ converging to $g\in W^{k-1,p}$, 
we know that $\bd F_j=G_j$, then
\begin{multline*}
\lVert F_j-f_j\rVert_{k,p}^p=\lVert \Aop(F_j-f_j)
\rVert_{L^p\oplus W^{k-1,p}}^p=\lVert \ba(F_j-f_j)\rVert_{p}^p+
\lVert \bd(F_j-f_j)\rVert_{k-1,p}^p
\\
=\lVert G_j-\tilde g_j\rVert_{k-1,p}^p\le\lVert G_j-g_j\rVert_{k-1,p}^p+
\|\Phi_{\epsilon_j}\|^p_{k-1,p}\lvert\langle g_j\rangle_\R\rvert^p 
\to0, 
\qquad \text{as}\quad j\to+\infty,
\end{multline*}
by the isometric properties of $\Aop$. If we let $F$ denote the abstract limit
of the Cauchy sequence $F_j$ in $W^{k,p}$, we conclude that $F=f$ in $W^{k,p}$. 
In conclusion, it did not matter whether we used the prescribed Cauchy 
sequence or a competitor, and hence $\bg g\in W^{k,p}$ is well-defined for 
$g\in W^{k-1,p}$.
%If $\{G_j\}_j$ with $G_j\in C^{k-1}_0(\R)$ would be another Cauchy sequence
%converging to $g$ in $W^{p}_{k-1}$, we would define $\tilde G_j$ using 
%possibly 
%another sequence $N(j)$ possible tending faster to infinity, and
%write $U_j$ as the primitive of $\tilde G_j$, with $U_j\in C^k_0(\R)$.   
%Putting $F_j:=U_j-\bb U_j\in W^{k,p}$, we obtain that $\ba F_j=0$ while
%$\bd F_j=U_j'=\tilde G_j$. Again by the isometric property of $\Aop$, 
%\[
%\|F_j-f_j\|_{k,p}^p
%=\|\tilde G_j-\tilde g_j\|^p_{k-1,p}\le\|G_j-g_j\|^p_{k-1,p}+
%\|\phi_{n(j)}\|_{k-1,p}^p|\langle g_j\rangle_\R|^p+
%\|\phi_{N(j)}\|_{k-1,p}^p|\langle G_j\rangle_\R|^p,
%\]
%where the last two terms tend to $0$ as $j\to+\infty$,
%which shows that the limit $\bg g=\lim_{j\to+\infty}f_j$ is well-defined in
%$W^{k,p}$ and independent of the choice of Cauchy sequence $g_j$. 
Finally, we observe from \eqref{eq-contractive1.23} that
\[
\|\bg g\|_{k,p}\le \|g\|_{k-1,p},\qquad g\in W^{k-1,p},
\]
which makes $\bg:W^{k-1,p}\to W^{k,p}$ a linear contraction. 

\begin{figure*}[h]
    \centering
    \includegraphics*[width=0.6\textwidth]{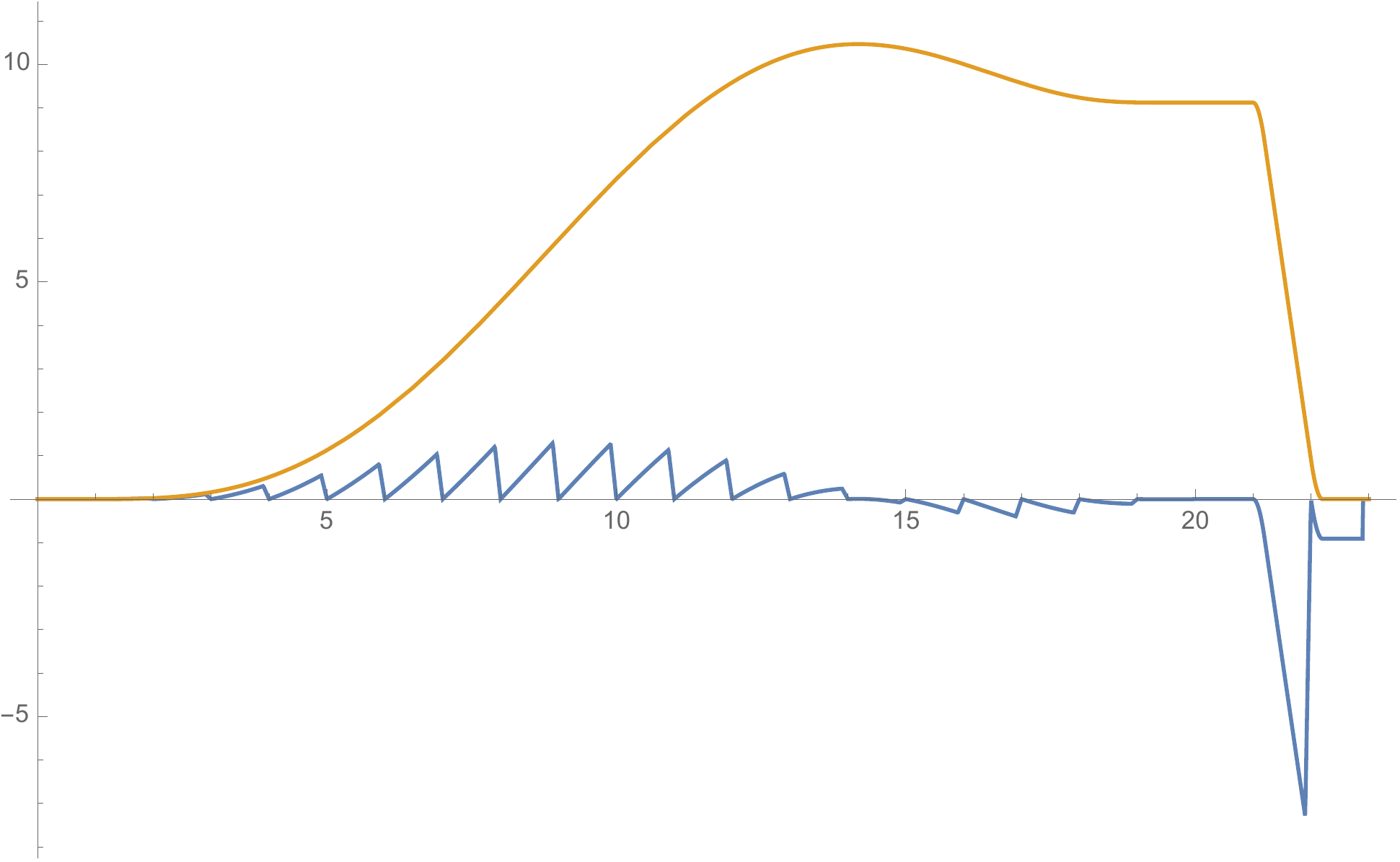}
    \caption{The functions $f_j$ and $u_j$, for $W^{1,p}$}
    \label{fig:beta}
\end{figure*}

\begin{proof}[Proof of Lemma, part 2]
We show that the remaining properties, those involving $\bg$: 
$\ba\bg=0$ and $\bd\bg=\id_{W^{k-1,p}}$. 
We read off from \eqref{eq-propabgd} that $\ba\bg g=0$ and $\bd\bg g=g$ 
for $g\in W^{k-1,p}$, which does it.
\end{proof}

\begin{rem}
A classical theorem by Day \cite{Day} states that the $L^p$-spaces for 
$0<p<1$ have trivial dual. It follows from Douady-Peetre's isomorphism 
$W^{k,p}\cong L^p\oplus\cdots\oplus L^p$ that the same is true for the 
Sobolev spaces $W^{k,p}$. We note here that any space that could be 
realised as a space of distributions would necessarily admit non-trivial 
bounded functionals (the test functions, for instance). 
Hence it follows that the elements of $W^{k,p}$ cannot even be interpreted
as distributions.
\end{rem}

\section{The smooth regime}
\label{sec-smooth}

\subsection{Classes of test functions} Although we mainly focus on  
the classes $\S_{p,\theta,\M}^\circledast$, we also mention 
the Hermite class $\S^{\mathrm{Her}}$ of weighted polynomials
\[
\S^{\mathrm{Her}}=\left\{f:\,\,f(x)=q(x)\e^{-x^2},\,\,
\text{ where }\, q\in \operatorname{Pol}(\R)\right\},
\]
where $\operatorname{Pol}(\R)$ denotes the linear space 
of all polynomials of a real variable. We show that $\S^{\mathrm{Her}}$ 
consists of $(p,0)$-tame functions. To get a class which fits into the 
associated $L^p$-Carleman class, we also intersect the Hermite class 
with the space $\{f\in C^{\infty}(\R):\lVert f\rVert_{p,\M}<+\infty\}$  to obtain
the $\S^{\mathrm{Her}}_{p,\M}$.

\begin{lem}\label{Lemma-Hermite}
The Hermite class $\S^\mathrm{Her}$ consists of $(p,0)$-tame functions.
\end{lem}

\begin{proof} 
Let $f\in \S^\mathrm{Her}$. Then $f(x)=q(x)\e^{-x^2}$ for some 
polynomial $q\in\operatorname{Pol}(\R)$. Let $d$ be the degree of $q$, 
and assume that all coefficients of the polynomial $q(x)$ are bounded by 
some number $m=m_q$. Let ${\mathbf L}$ be the operator on 
$\operatorname{Pol}(\R)$ defined by the relation
\[
{\mathbf L}q(x)=\e^{x^2}\frac{\diff}{\diff x}
\left(q(x)\e^{-x^2}\right)=q'(x)-2x q(x),\quad q \in \operatorname{Pol}(\R).
\]
It follows that the Taylor coefficients 
$\widehat{{\mathbf L}q}(j)$ of ${\mathbf L}q$ can be 
estimated rather crudely in terms of $d$ and $m_q$:
\begin{equation}
\label{est-L-coeff}
\lvert\widehat{{\mathbf L}q}
(j)\rvert
\leq (d+2)m_q,\qquad j=0,1,\ldots,d+1,
\end{equation}
and the coefficients vanish for $j>d+1$, so that the degree of ${\mathbf L}q$
is at most $d+1$.
By repeating the same argument, with ${\mathbf L}q$ in place of $q$, 
we obtain 
\[
\lvert \widehat{{\mathbf L}^2q}(j)\rvert 
\leq (d+3)(d+2)m_q,
\]
and, by iteration of \eqref{est-L-coeff}, it follows more generally that
for $j=0,1,2,\ldots$,
\begin{equation}
\label{est-Ln-coeff}
\lvert\widehat{{\mathbf L}^nq}(j)\rvert
\leq (d+2+n-1)(d+n-2)\cdots(d+2)m_q=(d+2)_{n} m_q,\quad n=1,2,\ldots,
\end{equation}
where we use the the standard Pochhammer notation $(x)_{k}=
x(x+1)\cdots (x+k-1)$ for $x\in\R$.

We proceed to estimate the $L^\infty$-norm of $f^{(n)}$. By the way 
${\mathbf L}$ was defined, we may estimate
\[
\|f^{(n)}\|_\infty = \sup_{x\in\R}|\e^{-x^2}{\mathbf L}^nq(x)|
=\sup_{x\in\R}\bigg\lvert 
\e^{-x^2}\sum_{j=0}^{d+n}\widehat{{\mathbf L}^nq}(j) 
x^{j}\bigg\rvert\leq \sum_{j=0}^{d+n}
\lvert\widehat{{\mathbf L}^nq}(j)\rvert 
\sup_{x\in\R}\,\lvert x\rvert^{j}\e^{-x^2}.
\]
By trivial calculus, the supremum on the right-hand side is attained at 
$|x|=\sqrt{j/2}$, and this we may implement in the above estimate 
while we recall the coefficient estimate \eqref{est-Ln-coeff}:
\begin{equation}
\label{est-fn}
\|f^{(n)}\|_\infty \leq m_q(d+2)_{n}\sum_{j=0}^{d+n}
\left(\frac{j}{2}\right)^{j/2}\e^{-j/2}\leq 
m_q(d+2)_{n+1}\left(\frac{d+n}{2}\right)^{(d+n)/2}\e^{-(d+n)/2},
\end{equation}
where in the last step we just estimated by the number of terms multiplied 
by the largest term.
Next, we take logarithms and apply elementary estimates to arrive at
\[
(1-p)^n\log \|f^{(n)}\|_\infty\leq (1-p)^n
\left\{\log m_q+(n+1)\log(d+n+2)+
\frac{d+n}{2}\left(\log \frac{d+n}{2}-1\right)\right\}.
\]
As the expression in brackets is is of growth order 
$\mathrm{O}\left((n+d)\log(n+d)\right)$, it follows that the right-hand
side expression tends to $0$ as $n\to+\infty$. 
It is then immediate that $f$ is $(p,0)$-tame, and since $f$ was an arbitrary
element $f\in \S^\mathrm{Her}$, the claim follows.
\end{proof}

Next, we turn to the question of what is required of the weight sequence 
$\M$ in order for $\S^\mathrm{Her}_{p,\M}$ to contain non-trivial 
functions. 
We thus need to estimate the $L^p$-norms of $f^{(n)}$ for 
$f\in \S^{\mathrm{Her}}$. 
By performing the same estimates as in the above proof and by 
appealing to the $p$-triangle inequality (which says that $(a+b)^p\leq a^p+b^p$ 
for $0<p\le1$  and positive $a$ and $b$), we see that
\[
\|f^{(n)}\|_p^p=\int_\R\lvert f^{(n)}(x)\rvert^p\diff x 
\leq m_q^p(d+2)_{n}^p
\sum_{j=0}^{d+n}\int_{\R}\lvert x\rvert^{pj}\e^{-px^2}\diff x.
\]
The integrals on the right-hand side are easily evaluated:
\[
\int_\R\lvert x\rvert^{pj}\e^{-px^2}\diff x=p^{-(pj+1)/2}
\Gamma((pj+1)/2),
\]
so that the above gives the estimate
\[
\|f^{(n)}\|_p^p\le m_q^p(d+2)_{n}^p
\sum_{j=0}^{d+n}p^{-(pj+1)/2}
\Gamma((pj+1)/2)\le m_q^p(d+n+1)(d+2)_{n}^p p^{-(pd+pn+1)/2}
\Gamma((pd+pn+1)/2).
\]
Next, using Stirling's formula yields the estimate that
\[
\Gamma((pd+pn+1)/2)=O\left((n!)^{p/2}n^{p(d-1)/2}
\left(\frac{p}{2}\right)^{pn/2}\right),
\]
whence
\[
m_q^p(d+n+1)(d+2)_n^pp^{-(pd+pn+1)/2}
\Gamma((pd+pn+1)/2) = O
\left(\frac{n^{1+\frac{3p(d-1)}{2}+2p}}{2^{pn}}(n!)^{p+\frac{p}{2}}\right)
\]
Ignoring the specific constants, we find that
 \[
\lVert f^{(n)}\rVert_p^p=O\left((n!)^{3p/2}\frac{n^\alpha}{\beta^n}\right),
\]
where $\alpha>0$ and $\beta>1$ are some constants, and hence it follows that
$\|f^{(n)}\|_p=O\left((n!)^{3/2}\right)$.
We do not proceed to analyse in full detail what $\M=\{M_n\}$ should have 
to fulfill in order for $\S_{p,M}^\mathrm{Her}$ to be a meaningful 
test class, but we note that the above implies that
$\S_{p,M}^{\mathrm{Her}}=\S^{\mathrm{Her}}$ 
if e.g. $\M=\{M_n\}_n$ meets $M_n\ge (n!)^{\sigma}$ for any
$\sigma\geq 3/2$, which we understand as a \emph{Gevrey class} condition.

\subsection{The smooth regime: Theorem~
\ref{thm-smoothmap}}\label{smooth_regime}

We have already presented the bootstrap argument, which is the main 
ingredient in the proof of 
Theorem~\ref{thm-smoothmap}, in the introduction.
We proceed to fill in the remaining details.

\begin{proof}[Proof of Theorem~\ref{thm-smoothmap}] 
We consider the canonical mapping $\bp: f\mapsto (f,f',f'',\ldots)$, 
defined initially from 
$\S_{p,\theta,\M}^\circledast$ into $\ell^\infty (L^p,\M)$. The quasinorm
on the sequence space $\ell^\infty (L^p,\M)$ is supplied in equation
\eqref{eq-quasinormlinfty}. 
Then, as a matter of definition, 
\[
\|\bp f\|_{\ell^\infty (L^p,\M)}=\|f\|_{p,\M},\qquad f\in
\S_{p,\theta,\M}^\circledast
\] 
and passing to the completion the mapping extends to an isometry 
$\bp:\, W_\M^{p,\theta}\to\ell^\infty (L^p,\M)$. 
In particular, the mapping $\bp$ is injective.
We recall that we think of an element $f\in W_\M^{p,\theta}$ as an abstract 
limit of a Cauchy sequence $\{f_j\}_j$ in the norm $\nl \cdot \nr_{p,\M}$, with 
$f_j\in \S_{p,\theta,\M}^\circledast$. 
For such an element, the mapping is obtained by taking the $L^p$-limit 
in all coordinates (that is, the sequence of higher order derivatives). 
This is well-defined, since if we were to take two Cauchy sequences 
$\{f_j\}$ and $\{\tilde{f}_j\}$ with the same abstract limit 
$f\in W_\M^{p,\theta}$, the images under the mapping would agree in each 
coordinate since 
\[
\big\| f_j^{(n)}-\tilde{f}_j^{(n)}\big\|_p\leq M_n\|  f_j-\tilde{f}_j\|_{p,\M}.
\]
Next, we show that the image of $W_\M^{p,\theta}$ under the mapping 
is actually inside 
$C^\infty\times C^\infty\times\ldots$. Indeed, since all the differences 
$f_j-f_k$ are in $\S_{p,\theta,\M}^{\circledast}$, it follows from 
\eqref{eq-basicineq1.9} that
\[
\big\|f_j^{(n)}-f_k^{(n)}\big\|_\infty \leq \bigg\{\e^{p(1-p)^{-n-1}\kappa(p,\M)}
\prod_{l=1}^{n}M_{l}^{-(1-p)^{l-n-1}p}\bigg\}\,\| f_j-f_k\|_{p,\M},
\]
where the right hand side tends to zero as $j,k\to{+\infty}$. It is now pretty
obvious that under the finite $p$-characteristic condition 
$\kappa(p,\M)<+\infty$, each function $f_j^{(n)}$ has a limit $g_n\in C(\R)$ 
as $j\to+\infty$. Moreover, as all the derivatives converge uniformly, 
we conclude that $g_n=g_{n-1}'=\ldots =g_0^{(n)}$ for each $n=0,1,2,\ldots$, 
and hence that
\[
\partial \bp_n f=\bp_{n+1}f,\qquad n=0,1,2,\ldots.
\]
From this relation it follows that the first coordinate map $\bp_0$ is 
injective. Indeed, if $\bp_0f=0$, then by iteration $\bp_n f=0$ for 
$n=0,1,2,\ldots$. Hence $\bp f=0$ and thus $f=0$ by the injectivity of $\bp$.

Finally, we show that the given test class is stable under the completion. 
We already saw that the limiting functions are of class $C^\infty$, and 
naturally $\|f\|_{p,\M}<+\infty$ for each $f\in W_\M^{p,\theta}$. 
What remains is to show that
\[
\limsup_{n\to+\infty} (1-p)^n\log \| f^{(n)}\|_\infty \leq \theta.
\]
However, the bound \eqref{eq-basicineq1.9} tells us that
\[
\|f^{(n)}\|_\infty\le\|f\|_{p,\M}\,\e^{\theta(1-p)^{-n}+p(1-p)^{-n-1}\kappa(p,\M)}
\prod_{j=1}^{n}M_{j}^{-(1-p)^{j-n-1}p},
\]
so that
\[
(1-p)^n\log\|f^{(n)}\|_\infty \leq \theta+\frac{p}{(1-p)}
\bigg\{\kappa(p,\M)-\sum_{j=1}^n(1-p)^{j}\log M_j\bigg\}+(1-p)^n\log\|f\|_{p,\M}.
\]
By Remark~\ref{rem:001} (b), it follows that 
\[
\kappa(p,\M)-\sum_{j=1}^n(1-p)^{j}\log M_j\to 0, \quad n\to{+\infty},
\]
and hence
\[
\limsup_{n\to+\infty}(1-p)^n\log\|f^{(n)}\|_\infty \leq \theta,
\]
as claimed. 
\end{proof}

\section{Independence of derivatives}
\label{sec-uncoupling}
\subsection{Some notation and preparatory material} 
We previously considered the mapping $\bp: f\mapsto (f, f', f'',\ldots)$. 
To conform with notation introduced earlier in this paper as well as in
the work of Peetre, we will use the letter $\ba$ instead of $\bp_0$ to 
refer to the first coordinate map $\ba: W^{p,\theta}_\M\to L^p$. Likewise, 
we write $\bd$ for the mapping $W_\M^{p,\theta}\to W_{\M_1}^{p,\theta(1-p)^{-1}}$
of taking the derivative (on the test functions). Here, we recall that $\M_1$ 
is the shifted sequence $\M_1=\{M_{n+1}\}_n$. More precisely, in terms of a 
Cauchy sequence $\{f_j\}_j$ of test functions in 
$\S_{p,\theta,\M}^\circledast$ which converges in quasinorm to an 
abstract limit $f\in W^{p,\theta}_\M$, we write 
\[
\ba f=\lim_j\ba f_j=\lim_j f_j\in L^p\quad\text{ and }\quad 
\bd f=\lim_j\bd f_j=\lim_j f_j'\in W_{\M_1}^{p,\theta(1-p)^{-1}}.
\]

As we saw in connection with the Sobolev $W^{k,p}$-spaces for $0<p<1$, the key 
step was the construction of the two lifts $\bb$ and $\bg$. If we may find the
analogous lifts in the present setting, the rest of the proof will carry over
almost word for word from the Sobolev space case.
It turns out that the lifts $\bb$ and $\bg$ are built in the same manner as 
before, using the existence of invisible mollifiers $\Phi_\epsilon$ analogous to 
the case of $W^{k,p}$ (compare with Lemma~\ref{lem:mollifier-finite}),
with slight technical obstacles in the construction of $\bg$,
related to the unbounded support of test functions.
%For brevity, we will show here how to build $\bb$ 
%(Proposition~\ref{prop-beta} below), and leave the details pertaining to 
%$\bg$ to the interested reader.

Before turning to the mollifiers, we make an observation regarding the weight 
sequence. We will make a dichotomy between the cases 
\[
\lim_{n\to{+\infty}} (1-p)^n\log M_n =0 \quad \text{and}\quad 
\liminf_{n\to{+\infty}} \log (1-p)^n\log M_n>0
\] 
in the definition of $p$-regularity. 
%We treat the former case below, i.e. we assume henceforth that 
%$\lim_{n\to{+\infty}}(1-p)^n \log M_n=0$. 
%The latter case is handled by finding a minorant sequence 
%$\wtM=\{\twtM_n\}$ for which 
%$$
%\lim_{n\to{+\infty}}(1-p)^n\log \twtM_n=0,
%$$ 
%with all other properties mentioned retained, and applying the 
%argument below to that minorant. 
%

\begin{lem}\label{lem:minorant}
Let $\M$ be a log-convex sequence such that $\kappa(p,\M)=+\infty$ and 
assume that 
\[
\liminf_{n\to\infty}(1-p)^n\log M_n>0.
\] 
Then there exists a $p$-regular minorant sequence $\wtM$ with 
$\kappa(p,\wtM)=+\infty$, such that 
\[
\lim_{n\to\infty} (1-p)^n\log\twtM_n=0.
\]
\end{lem}

For more complete details of this proof, see \cite{BehmWennman}, 
as well as Behm's thesis \cite{BehmThesis}, where the former paper is 
discussed. 

\begin{proof}[Proof sketch]
Let $c_0=\liminf _{n\to+\infty}(1-p)^n\log M_n>0$. Then
\[
c_0(1-p)^{-n}\leq (1+\mathrm{o}(1))\log M_n,
%c_0\leq (1-p)^n\log M_n+\mathrm{o}(1),
\]
so that
\[
\frac{c_0}{n+1}(1-p)^{-n}\leq \frac{1}{n+1}(1+\mathrm{o}(1))\log M_n
\le \log M_n+\mathrm{O}(1).
%c_0\leq (1-p)^n\log M_n+\mathrm{o}(1),
\]
We now put 
\[
\twtM_n=C_1\e^{c_0(n+1)^{-1}(1-p)^{-n}},
\]
where the positive constant $C_1$ is chosen such that 
$\twtM_n\leq M_n$ for all integers $n\geq 0$. 
It follows that $\wtM=\{\twtM_n\}_n$ is $p$-regular with 
$\kappa(p,\wtM)=+\infty$, and that $(1-p)^n\log \twtM_n=\mathrm{o}(1)$,
as needed.
\end{proof}

\subsection{The construction of mollifiers}
The invisibility lemma runs as follows. 

\begin{lem}[invisibility lemma for $L^p$-Carleman spaces]\label{mollifier}
Assume that $\M$ is $p$-regular and has infinite $p$-characteristic 
$\kappa(p,\M)=+\infty$. 
For any $\epsilon>0$ there exists a non-negative function 
$\Phi\in \S_{p,0,\M}^\circledast$ (with $\theta=0$) such that
\[
\int_{\R} \Phi\dd x=1, \qquad \supp \Phi\subset [0,\epsilon], 
\qquad \lVert \Phi\rVert_{p,\M}<\epsilon.
\]
\end{lem}
%In order to prove Lemma~\ref{mollifier} we need to, 
%for a given $\epsilon>0$, 
%find a sequence $a=\{a_{l}\}_{l\geq 1}$ such that $\nl a\nr_{\ell^1} 
%<\epsilon$, 
%and $a$ satisfies \eqref{eq:separation}, 
%\eqref{limsupsequence} and \eqref{normsequence} with this given $\epsilon$. 

The tools needed to prove the lemma were developed back in 
Section~\ref{ss:convol},
in connection with the Invisibility Lemma for $W^{k,p}$. For a sequence 
$\{a_l\}_{l=1}^\infty$
we form the associated convolution product $\Phi_{1,\infty}$. We recall that
$\Phi_{1,\infty}$ has support $[0,\sum_{l\geq 1} a_l]$,
and enjoys the estimates \eqref{eq-elest1.002} and \eqref{eq-elest1.004}, 
provided that
$(a_j)_{j\geq 1}$ is a decreasing $\ell^1$-sequence. 

In order to bound the support of $\Phi^{(n)}$, will need to estimate the 
sums $\sum_{j\geq n+1}a_j$. 
If we write $a_j=c_j \alpha_j$ and ask for $\{\alpha_j\}_{j}$ to be a
decreasing sequence of positive numbers, and that $\{c_j\}_j\in \ell^1$, 
we obtain the bound
\begin{equation}\label{eq:separation}
\sum_{j=n+1}^{+\infty}a_j\leq c \alpha_{n+1},\qquad n\geq 0,
\end{equation}
where $c=\lVert \{c_j\}_{j}\rVert_{\ell^1}$.
Hence the right-hand side of \eqref{eq-elest1.004} may be estimated further:
\[
\lVert \Phi_{j,\infty}^{(n)}\rVert_p^p\leq 2^n 
\frac{c\alpha_{n+1}}{(a_1\cdots a_{n+1})^p}
=\frac{2^nc}{(c_1\cdots c_{n+1})^p}\frac{\alpha_{n+1}}{(\alpha_1\cdots 
\alpha_{n+1})^p}.
\]
We summarise what we ask for a sequence $a=\{a_l\}_{l}=\{c_l\alpha_l\}_{l}$ 
to satisfy
in to guarantee that $\Phi_{1,\infty}=\Phi_{1,\infty,a}$ satisfies the conclusion 
of the lemma. 
First, for the estimates \eqref{eq-elest1.002} and \eqref{eq-elest1.004} 
to come into play, 
we need \eqref{eq:separation}, which is satisfies as soon as 
$\{c_j\}_{j}\in \ell^1$ and 
$\{\alpha_j\}_j$ is decreasing. The $(p,0)$-tameness is, 
in view of \eqref{eq-elest1.002}, ensured by
\begin{equation}\label{limsupsequence}
\limsup_{n\to\infty}(1-p)^n\log \frac{1}{a_1\cdots a_n}\leq 0.
\end{equation}
In order to verify the norm control 
$\lVert \Phi_{1,\infty}\rVert_{p,\M}<\epsilon$, 
we shall require
\begin{equation}\label{normsequence}
\frac{2^nc}{(c_1\cdots c_{n+1})^p}
\frac{\alpha_{n+1}}{(\alpha_1\cdots \alpha_{n+1})^p}\leq 
\epsilon^pM_n^p,\qquad n\geq 0.
\end{equation}
Moreover, the assertion that $\supp \Phi_{1,\infty}\subset [0,\epsilon]$ is 
equivalent to having $\lVert a\rVert_{\ell^1}\leq\epsilon$,
which follows from the requirement $\alpha_1\leq \epsilon/c$.

Let us briefly discuss the structural consequences for sequences 
$\{\alpha_l\}_l$ 
satisfying a requirement of the form
\[
\frac{\alpha_{l+1}}{(\alpha_1\cdots \alpha_{l+1})^p}\leq DM_l^r,
\]
where $D$ and $r$ are positive constants. 

Due to the product structure, we make a quotient ansatz $\alpha_l=b_l/b_{l+1}$.
In terms of $\{b_l\}_{l\geq 1}$, the above relations read
\[
\frac{b_{l+1}}{b_1^p b_{l+2}^{1-p}}\leq DM_l^r,\qquad l\geq 0.
\]
Since $a_1$ is determined by two numbers $b_i, i=1,2$, we are free to 
choose $b_1$. To simplify matters we fix $b_1^pD=1$.
Then in the extremal case where we have equality for $l\geq 0$,
\[
\frac{b_{l+1}}{b_{l+2}^{1-p}}=M_l^r
\]
this gives iteratively the sequence $\{b_l\}_l$ as
\[
b_{l+2}=(b_{l+1}\Lambda_l)^{1/(1-p)}, \qquad \Lambda_l=M_l^{-r}.
\]
for $l\geq 0$. In terms of $b_1$ and $\{\Lambda_j\}_j$, we obtain 
$b_{l+2}=(b_1\Lambda_0\Lambda_1^{1-p}\cdots\Lambda_n^{(1-p)^l})^{(1-p)^{-l-1}}$, or
\[
b_{l+2}=\left(b_1 M_0^{-r}M_1^{-r(1-p)}\cdots M_n^{-r(1-p)^l}\right)^{(1-p)^{-l-1}}
\]
In case $\kappa(p,\M)$ would have been finite, we could put 
$b_1=\prod_{j\geq 0} M_{j}^{r(1-p)^j}$ and obtain
\begin{equation}\label{eq:div-prod}
b_{l}=\prod_{j\geq 0}M_{l+j-1}^{r(1-p)^j},\qquad n\geq 1.
\end{equation}
Since in the case at hand, we have $\kappa(p,\M)=+\infty$, 
such a choice will not work directly. 
However, it gives a hint as to how to chose these numbers.

With these preliminaries, we begin our construction. 

\begin{proof} By Lemma~\ref{lem:minorant}, we may assume without loss of 
generality that 
\[
\lim_{n\to+\infty}(1-p)^n\log M_n=0.
\]
If $\M$ violates this, then we replace $\M$ by a minorant $\wtM$, and 
proceed as below.
The result then follows since $\S^\circledast_{p,0,\wtM}
\subset \S^\circledast_{p,0,\M}$,
where the inclusion is contractive in quasinorm.

By the assumptions of $p$-regularity the number
\[
\delta:=\liminf_{n\to\infty}\frac{\log M_n}{n\log n}
\]
satisfies $1<\delta\leq  +\infty$. Let $r$ and $\rho$ be numbers 
with $0<r\leq p$ and $\rho>0$, such that
\[
(p-r)\geq \delta^{-1}(1+\rho)p,
\]
which is possible since $\delta>1$.
Led by the above considerations, we consider instead of 
\eqref{eq:div-prod} the truncated products
\[
b_{l,k}=\prod_{j=0}^k M_{l-1+j}^{r(1-p)^{j}},\quad l\geq 1
\]
where $k$ remains to be chosen. Next, we write
\[
\alpha_{l,k}=\frac{b_{l,k}}{b_{l+1,k}},\qquad c_l=\frac{1}{l^{1+\rho}},
\qquad l\geq 1,
\]
and recall that we want $a_{l,k}=c_l\alpha_{l,k}$.
Since $\M$ is assumed to be logarithmically convex,
\[
\alpha_l=\prod_{j=0}^k\left(\frac{M_{l-1+j}}{M_{l+j}}\right)^{p(1-p)^{j+1}}\leq 
\prod_{j=0}^k\left(\frac{M_{l-2+j}}{M_{l+j-1}}\right)^{p(1-p)^{j+1}}=\alpha_{l-1},
\quad l\geq 2.
\]
It follows that $\{a_l\}_l$ is a decreasing sequence, so since 
$\{c_l\}_l\in \ell^1$, the above discussion is applicable. In particular, 
the support of $\Phi$ is included in $[0,c \alpha_{1,k}]$, which we may 
calculate as
\[
c a_{1,k}=c \frac{b_{1,k}}{b_{2,k}}=c\frac{M_0^{r}}{ M_{k+1}^{r(1-p)^{k}}} 
\frac{1}{\prod_{j=1}^k M_{j}^{pr(1-p)^{j-1}}}.
\]
Since $\kappa(p,\M)=+\infty$ and since $M_{k+1}\ge 1$ for large enough $k$, 
this tends to zero as $k\to+\infty$.

To verify that $\Phi$ is $(p,0)$-tame, we observe that
\begin{align*}
(1-p)^n\log \lVert \Phi^{(n)}\rVert_\infty&\leq (1-p)^n\log
\left(a_1\cdots a_{n+1}\right)^{-1} 
\\&=(1-p)^n\log \frac{b_{n+2,k}}{b_{1,k}}+(1+\rho)(1-p)^n\sum_{j=1}^n\log j
\end{align*}
so that if $k$ is large enough for $b_{1,k}\geq 1$ to hold,
\[
(1-p)^n\log\lVert \Phi^{(n)}\rVert_\infty\leq r(1-p)\sum_{j=0}^k(1-p)^{j+n}
\log M_{n+j}+o(1)=o(1),
\]
since $k$ is held fixed when $n\to\infty$, and $(1-p)^n\log M_n\to 0$.

We turn to demonstrate \eqref{normsequence}. We first observe that
\[
2^{n}c\frac{1}{(c_1\cdots c_{n+1})^p}=\e^{f(n)},
\]
where $f(n)=\log c + n\log 2+ p(1+\rho)\log[(n+1)!]$. Moreover, it holds that
\[
\frac{\alpha_{n+1,k}}{(\alpha_{1,k}\cdots \alpha_{n+1,k})^p}=
\frac{\frac{b_{n+1,k}}{b_{n+1,k}}}{\left(\frac{b_{1,k}}{b_{n+2,k}}\right)^p}=
\frac{b_{n+1,k}}{b_{1,k}^pb_{n+2,k}^{1-p}}=
\frac{1}{b_{1,k}^p}\frac{1}{M_{n+k+1}^{r(1-p)^{k}}M_{n}^{p-r}}M_n^{p}
\]
Putting this together while esimating $M_{n+k+1}^{r(1-p)^k}\geq 1$, we 
obtain the estimate
\[
\lVert \Phi_{1,\infty}^{(n)}\rVert_p^p\leq \frac{1}{b_{1,k}^p}
\e^{f(n)-(p-r)\log M_n}M_n^p.
\]
In particular, observe that 
\[
f(n)-(p-r)\log M_n\leq p(1+\rho) n\log n-\delta^{-1}\log M_n+\mathrm{O}(1)= 
\mathrm{O(1)}
\]
is bounded above by some constant $K$ independently of $k$. 
In terms of $\lVert \Phi_{1,\infty}^{(n)}\rVert_p^p$, this means that
\[
\lVert \Phi_{1,\infty}^{(n)}\rVert_p^p\leq \frac{\e^K}{b_{1,k}^p}M_n^p.
\]
Since $b_{1,k}\to+\infty$, as $k\to\infty$, and since $K$ is independent of $k$, 
the estimate $\lVert \Phi_{1,\infty}\rVert_\M$ follows if we choose $k$ 
large enough.
This finished the proof.
\end{proof}

\subsection{Construction of $\bb$ and $\bg$}
\label{subsec-bbbg}
The existence of invisible mollifiers finally allows us to define the 
lifts $\bb:L^p\to W^{p,\theta}_\M$ and $\bg: W_{\M_1}^{p,\theta_1}\to W_{\M}^{p,\theta}$.

\begin{prop}
\label{prop-beta}
Suppose $\M$ is $p$-regular with $\kappa(p,\M)=+\infty$. Then there exists
a continuous linear mapping $\bb: L^p\to W_\M^{p,\theta}$ such that
\[
\ba\circ\bb=\operatorname{id}_{L^p} \quad\text{and}\quad \bd\circ\bb=0.
\]
\end{prop}

\begin{proof}
%We begin with the mapping $\bb$. 
It is enough to demonstrate the result for $\theta=0$, since the general 
case then follows by inclusion.
The definition procedure remains the same as that in Section~\ref{ss:lift-b}. 
We denote by $\mathscr{F}_p$ the space of test functions, and to each 
$g\in\mathscr{F}_p$ 
we aim to associate a Cauchy sequence 
$\{g_j\}_j\subset\S_{p,0,\M}^{\circledast}$ such that
\begin{equation}\label{eq:rel-b}
\lim_{j\to+\infty}\lVert g-g_j\rVert_p=0\quad 
\lim_{j\to+\infty}\lVert g_j'\rVert_{p,\M_1}=0,
\end{equation}
where $\M_1$ is the shifted sequence $\M_1=\{M_{j+1}\}_{j\geq 0}$. 
We then declare $\bb g$ to be the abstract limit $\lim g_j$ in $W_{\M}^{p,0}$.
If this can be achieved, these properties show that $\bb$ is a rescaled 
isometry on $\mathscr{F}_p$:
Indeed, the norm splits as
\begin{equation}\label{eq:norm-id}
\lVert f\rVert_{p,\M}=\max \Bigg\{\frac{\lVert f\rVert_p}{M_0},\; 
\lVert f\rVert_{p,\M_1} \Bigg\},
\end{equation}
so we may write
\[
\lVert \bb g\rVert_{p,\M}^p=\lim_{j\to+\infty}\lVert g_j\rVert_{p,\M}^p=
\lim_{j\to+\infty}\max\Bigg\{\frac{\lVert g_j\rVert_p^p}{M_0^p},\;
\lVert g_j'\rVert_{p,\M_1}^p\Bigg\}=
\frac{\lVert g_j\rVert_p^p}{M_0^p}
\]
and they uniquely determine $\bb g$. Indeed, any other Cauchy sequence 
$\{\tilde{g}_j\}_j$ for which \eqref{eq:rel-b} holds will be equivalent 
to $\{g_j\}$ in $W_{\M}^{p,0}$:
\[
\lVert g_j-\tilde{g}_j\rVert_{p,\M}^p\leq \frac{\lVert g-g_j\rVert_p^p}{M_0}+
\frac{\lVert g-\tilde{g}_j\rVert_p^p}{M_0}+\lVert g_j'\rVert_{p,\M_1}^p+\lVert 
\tilde{g}_j'\rVert_{p,\M_1}^p
=o(1),
\]
as $j\to+\infty$. Moreover, \eqref{eq:rel-b} clearly shows that
\[
(\ba\circ\bb) g=g \quad\text{and}\quad (\bd\circ\bb)g=g,\quad 
g\in\mathscr{F}_p,
\]
and since these properties are stable under extension by continuity to $L^p$,
the proof is complete once the existence of a Cauchy sequence $\{g_j\}_j$ 
satisfying \eqref{eq:rel-b} is demonstrated.

The properties of $\M$ needed to apply Lemma~\ref{mollifier} are 
inherited by $\M_1$, 
so for each $n$ we may apply the lemma to 
$W_{\M_1}^{p,0}$ in order to produce 
functions $\Phi_n\in \S_{p,0,\M_1}^{\circledast}$ with 
support in $[0,\epsilon_n]$ such that $\|\Phi_n\|_{\M_1}<\epsilon_n$.
We set $g_j=1_{[a,b]}*\Phi_j$, so $g_j\in \S_{p,\theta,\M}^{\circledast}$.

Observe that for $n\geq 1$,
\[
(1_{[a,b]}*\Phi_j)^{(n)}=(\delta_a-\delta_b)*\Phi_j^{(n-1)}.
\]
It follows that
\begin{equation}\label{eq:moll-norm}
\lVert g_j^{(n)}\rVert_p=\nl (1_{[a,b]}*\Phi_j)^{(n)}\nr_p
\leq 2^{1/p}\lVert \Phi_j^{(n-1)}\rVert_p\leq 
2^{1/p}\epsilon_j M_n.
\end{equation}
It is clear that $g_j\to 1_{[a,b]}$ in $L^p$ (e.g. since it converges 
in $L^\infty$), 
so $\lim_j \nl g_j-g\nr_p=0$. 
It also follows from \eqref{eq:moll-norm} that 
\[
\lim_{j\to+\infty}\lVert g_j'\rVert_{p,\mathcal{M}_1}^p=
\lim_{j\to+\infty}\sup_{n\geq 1}\frac{\lVert g_j^{(j+1)}\rVert_p}{M_n}\leq 
\lim_{j\to+\infty}2^{1/p}\epsilon_j=0.
\]
That these considerations show that $\{g_j\}_j$ is a $W^{p,\theta}_\M$-Cauchy 
sequence in $\S_{p,\theta,\M}^\circledast$ follows similarly. 
We thus declare this limit to be our $\bb g$. Note that
\[
\lVert g_j\rVert_{p,\M}=\sup_{n\geq 0}\frac{\lVert g_j^{(n)}\rVert_p}{M_n}\to 
\frac{\lVert g\rVert_p}{M_0},\quad j\to\infty
\]
in light of the above calculations. Thus, $\bb$ is a densely defined 
bounded operator.
After extending $\bb$ to the entire space, we obtain a well-defined 
linear operator 
$\bb:L^p\to W_\M^{p,\theta}$ such that $\ba\bb=\id$ and $\bd\bb=0$.
\end{proof}

Recall the notation $\M_1=\{M_{k+1}\}_{k}$ and $\theta_1=\theta/(1-p)$.
%To remedy this, we have the following lemma.
%\begin{lem}\label{lem:mollifier-2}
%For each $\epsilon>0$, let $\Phi_\epsilon$ be the function given by 
%lemma~\ref{mollifier}. Then
%for any $g\in S_{p,0,\M}^\circledast$, 
%$$
%\lim_{\epsilon\to 0}\,\lVert \Phi_\epsilon* g- g\rVert_{p,\M}=0.
%$$
%\end{lem}
\begin{prop}\label{prop-gamma}
Assume that $\M$ is $p$-regular with $\kappa(p,\M)=+\infty$. Then there exists
a continuous linear mapping $\bg: W_{\M_1}^{p,\theta_1}\to W_{\M}^{p,\theta}$ 
such that
\[
\bd\circ\bg=\operatorname{id}_{W_{\M_1}^{p,\theta_1}}\quad\text{and}
\quad\ba\circ\bg=0.
\]
\end{prop}
\begin{proof}
It is enough to construct $\bg$ on the class $\S_{p,\theta_1,\M_1}^\circledast$ 
which by 
definition is dense in $W_{\M_1}^{p,\theta_1}$. For any 
$g\in \S_{p,\theta_1,\M_1}^\circledast$,
we will show that there exists a Cauchy sequence $\{u_j\}_j$ in 
$\S_{p,\theta,\M}^\circledast$ 
with the following properties:
\begin{equation}\label{eq:prop-g}
\lim_{j\to+\infty}\lVert u_j\rVert_p^p=0,\quad 
\lim_{j\to+\infty}\lVert u_j'-g\rVert_{p,\M_1}^p=0,
\end{equation}
and declare $\bg g$ to be the abstract limit $\lim_{j\to+\infty} u_j$.
By the norm identity \eqref{eq:norm-id} and exactly the same argument as in 
Section~\ref{ss:lift-g}, it follows that $\bg$ is well-defined, and extends 
to a continuous linear map $\bg:W_{\M_1}^{p,\theta_1}\to W_{\M}^{p,\theta}$ 
which satisfies \eqref{eq:prop-g}.

The existence of the Cauchy sequence $\{u_j\}_j$ will be established 
if we for each given $\varepsilon>0$ can supply a function 
$u=u_\varepsilon\in \S_{p,\theta,\M}^\circledast$ such that
\[
\lVert u\rVert_p^p\leq C M_0\varepsilon,\quad 
\lVert u'-g\rVert_{p,\M_1}^p\leq C\varepsilon,
\]
for some positive constant $C$, which does not depend on $\varepsilon$. 
Indeed, if $\{\varepsilon_j\}_j$ is a sequence tending to zero, then 
the corresponding sequence $\{u_{\varepsilon_j}\}_j$ satisfies 
$\lVert u_{\varepsilon_j}-u_{\varepsilon_k}\rVert_{p,\M}^p\leq 
2C\max\{\varepsilon_j,\varepsilon_k\}$ and \eqref{eq:prop-g}, 
and is consequently a Cauchy sequence with the desired properties.

We cannot copy the argument from the Sobolev space case word-by-word, 
since the integration procedure employed in Section~\ref{ss:lift-g} is 
not sure to take $h\in \S_{p,\theta_1,\M_1}^\circledast$ into an element 
of $\S_{p,0,\M}^{\circledast}$, even if $\int_\R h(x)\diff x=0$. The main obstacle 
is the need to have $p$-integrability after taking the primitive.
Instead we refine the correction procedure. 
Define $f(x)$ by
\[
f(x)=\int_{-\infty}^x g(t)\diff t,
\]
which is well defined due to the estimate
\[
\lVert g\rVert_1\leq \lVert g\rVert_{p}^p\lVert g\rVert_{\infty}^{1-p}.
\]
For $f\in C(\overline{I})$, we introduce the notation
\[
\langle f\rangle^{\mathrm{osc}}_{I}:=\sup_{x\in I}\bigg\lvert f(x)-
\frac{1}{\lvert I\rvert}\int_{I}f(t)\diff t\bigg\rvert.
\]
We proceed to define a partition of $\R$ into intervals 
$I_n:=[\eta_n,\eta_{n+1}]$ 
with $\ldots \leq \eta_{-1}\leq \eta_{0}\leq \eta_{1}\leq \ldots$, such that 
\[
\sum_{n\in\Z} \lvert I_n\rvert (\langle f\rangle^{\mathrm{osc}}_{I_n})^p
\lesssim \epsilon.
\]
To see how this can be done, observe that $f'=g$ is uniformly bounded, 
so $f$ is uniformly continuous. 
Starting with with the partition induced by $\mathbb{Z}$,
i.e. $I_{m,1}^{\,1}=[m,m+1]$, where $m\in\Z$, we aim to repeatedly cut intervals 
in half until the requirement is satisfied.
For any interval $I_{m,1}^{\,1}$, if 
\[
(\langle f\rangle^{\mathrm{osc}}_{I_{m,1}^{\,1}})^p\leq \epsilon 2^{-\lvert m\rvert}.
\]
we do nothing.
If this inequality fails, consider the dyadic children 
$\{I_{m,j}^{\,k}\}_{j=1}^{2^{k}}$ 
of generation $k$ of $I_{m,1}^{\,1}$, 
where $k=k_m$ is large enough enough for 
\[
\big(\langle f\rangle^{\mathrm{osc}}_{I_{m,j}^{\,k}}\big)^p
\leq \epsilon 2^{-\lvert m\rvert},
\quad j=1,\ldots, 2^k
\]
to hold. That such a $k_m$ exists follows from the uniform continuity of $f$.
Repeating this procedure for all $m\in\mathbb{Z}$, we end up with intervals 
$\{I_{m,j}^{k_m}\}_{m,j}$ where $m\geq 1$ and $1\leq j\leq 2^{k_m}$, such that
\[
\sum_{m,j}\lvert I_{m,j}^{\,k_m}\rvert 
\big(\langle f\rangle^{\mathrm{osc}}_{I_{m,j}^{\,k_m}}\big)^p\leq 
\sum_{m\in \Z}\epsilon 2^{-\lvert m\rvert}\sum_{j=1}^{2^{k_m}}\lvert I_{m,j}^{k_m}\rvert 
\leq 3\epsilon.
\]
We may then order the intervals and index them by the single index $n\in\Z$, 
so that the right end-points are increasing. This is our partition $I_n$.

For each $n$, we may find subintervals $J_{n,1}=[\eta_n,\eta_n+\delta_n]$ and 
$J_{n,2}=[\eta_{n+1}-\delta_n, \eta_{n+1}]$ inside $I_n$, and functions 
$\Phi_{n,1}$ and $\Phi_{n,2}$ in $\S_{p,\theta_1,\M_1}^{\circledast}$ such that 
$\Phi_{n,1}$ and $\Phi_{n,2}$ are supported on $J_{n,1}^\circ$ and $J_{n,2}^\circ$, 
respectively, with norms $\lVert \Phi_{n,1}\rVert_{p,\M}=\lVert 
\Phi_{n,1}\rVert_{p,\M}\leq \epsilon_n$, 
where these numbers are such that
\[
\sum_{n}\delta_{n}=:\epsilon.
\]
and
\[
\sum_n \epsilon_n^p\leq \epsilon.
\]
Indeed, choosing first the sequence $\{\delta_{n}\}_n$ small enough in 
$\ell^1$ to ensure the first summability property while requiring 
$\delta_n< 1$, and then choosing 
$\epsilon_n\leq \delta_{n}^{p^{-1}}$, we may then use Lemma~\ref{mollifier} 
and a simple translation to find functions $\Phi_{n,1}$ and $\Phi_{n,2}$ 
supported on the right 
intervals with the right norm bounds. Let $\phi(x)$ denote the function 
gives by
\[
\phi(x)=\sum_{j\in\mathbb{Z}}\lambda_{n}[\Phi_{n,1}(x)-\Phi_{n,2}(x)],
\]
where
\[
\lambda_{n}=\frac{1}{\lvert I_n\rvert}\int_{I_n}f(t)\diff t,\quad i=1,2.
\]
Define $u\in C^\infty(\R)$ by
\[
u(x)=f(x)-\int_{-\infty}^x \phi(t)\diff t=\int_{-\infty}^x(g(t)-\phi(t))\diff t.
\]
We claim that $u\in L^p$ with small norm, and that $u'-g$ is small in 
$W_{\M_1}^{p,\theta_1}$. Indeed,
if $x\in I_n\setminus(J_{n,1}\cup J_{n,2})$, then
\begin{align*}
\big\lvert f(x)-\int_{-\infty}^x\phi(t)\diff t\big\rvert^p&
=\big\lvert f(x)-\sum_{m<n}\lambda_m\int_{J_{m,1}\cap J_{m,2}}
\left(\Phi_{m,1}(t)-\Phi_{m,2}(t)\right)\diff t-
\lambda_n\int_{J_{n,1}}\Phi_{m,1}(t)\diff t\big\rvert^p
\\
&=\big\lvert f(x)-\lambda_n\int_{J_{n,1}}\Phi_{m,1}(t)\diff t\big\rvert^p
=\big\lvert  f(x)-\frac{1}{\lvert I_n\rvert}\int_{I_n}f(t)\diff t\big\rvert^p ,
\end{align*}
and similarly if $x\in J_{n,i}$, then 
\[
\big\lvert f(x)-\int_{-\infty}^x\phi(t)\diff t\bigg\rvert^p\leq 
\big(\lvert f(x)\rvert + \lambda_n\Big 
\lvert \int_{\eta_n}^x(\Phi_{n,1}(t)-\Phi_{n,2}(t))\diff t\Big\rvert\bigg)^p
\leq 2^p\lVert f\rVert_\infty^p,
\]
which holds since the mean of $f$ on any subinterval is bounded 
by the essential supremum. It follows that
\begin{align*}
\lVert u\rVert_p^p\leq 
\sum_{n\in\mathbb{Z}}\left(2^{1+p}\lVert f\rVert_\infty\delta_n + 
\lvert I_n\rvert \sup_{x\in I_n}\big\lvert f(x)
-\frac{1}{\lvert I_n\rvert}\int_{I_n}f(t)\diff t\big\rvert^p\right) 
\\ 
\leq 2^{1+p}\lVert f\rVert_\infty^p\sum_{n\in\mathbb{Z}}\delta_n +
\epsilon\leq \left(2^{1+p}\lVert f\rVert_\infty^p+1\right)\epsilon.
\end{align*}
This proves that $\lVert u\rVert_p^p=O(\epsilon)$.

Turning to derivatives, it is clear that
\[
g-u'=g-(g-\phi)=\phi.
\]
Thus,
\[
\lVert u'-g\rVert_{p,\M_1}^p=2\sum_{n\in\mathbb{Z}}\lVert 
\Phi_{\epsilon_n}\rVert_p^p\leq 
2\sum_{n\in\mathbb{Z}} \epsilon_n^p\leq 2\epsilon.
\]
It thus follows that also $\lVert u'-g\rVert_p^p=O(\epsilon)$, 
and the assertion follows.
\end{proof}

\begin{proof}[Sketch of proof of Theorem~\ref{thm-uncoupling}] 
In principle, we just follow the algebraic approach from the proof of
Theorem \ref{Peetre} supplied in Subsection \ref{subsec-isomconstr}. 
So, we define a linear operator 
$\Aop:\,W^{p,\theta}_\M\to L^p\oplus W^{p,\theta_1}_{\M_1}$ given by 
$\Aop f:=(\ba f,\bd f)$ and analyze it using the properties of the maps $\ba,
\bb, \bd, \bg$. 
\end{proof}
\subsection{A remark on the size of $W_\M^{p,\theta}$.}
We conclude this section with a proposition which illustrates the 
size of $W_\M^{p,\theta}$ when $\kappa(p,\M)=+\infty$.

\begin{prop}
\label{prop:inclusion-tail-decay} 
In the setting of Theorem~\ref{thm-uncoupling}, we have the inclusions
\[
c_0(L^p,\M)
\subset \bp W_\M^{p,\theta} \subset 
\ell^\infty(L^p,\M)
\]
\end{prop}

\begin{proof}
The second inclusion is clear. To prove the first, let
$(f_0, f_1, f_2,\ldots)$ be an arbitrary element of $c_0(L^p,\M)$. 
By iterating Theorem~\ref{thm-uncoupling}, for each fixed 
$k$ there exists an element $F_k\in W_{\M}^{p,\theta}$ such that 
$\bp(F_k)=(f_1,\cdots, f_k,0,0,\ldots)$. 
We claim that $\lim F_j$ is an element of $W_{k,p}$, and that 
$\bp(\lim F_j)=(f_0,f_1,\ldots)$.
Indeed, if $k>j$, then
\[
\lVert F_j-F_k\rVert_{p,\M}=\sup_{j+1\leq l\leq k}
\frac{\lVert f_l\rVert_p}{M_l}=o(1),\quad \text{as\;}j\to\infty,
\]
since $(f_j)_j\in c_0(L^p,\M)$.
Thus $\{F_j\}$ is a Cauchy sequence in $W_\M^{p,\theta}$, and has 
an abstract limit.
Moreover,
\[
\bp(F_j)-(f_0,f_1,\ldots)=(0,\ldots,0, f_{j+1}, f_{j+2},\ldots),
\]
so it follows immediately that
\[
\lVert \bp(F_j)-(f_0,f_1,\ldots)\rVert_{c_0(L^p,\M)}
=\sup_{l\geq j+1}\frac{\lVert f_l\rVert_p}{M_l}
=o(1),\quad \text{as\;}j\to\infty.
\]
This proves our assertion.
\end{proof}

\section{The quasianalyticity transition and the 
parameter $\theta$}
\label{sec-carleman}

\subsection{A remark on the case $1\le p<+\infty$}

It is of course a basic question is if the Denjoy-Carleman theorem 
remains valid in the setting of the 
$L^p$-Carleman Classes in the parameter range $1\leq p<+\infty$ (without
any tameness requirement of course). This is of course true and should be
well-known, but we have unfortunately not been able to find a suitable 
references for this fact. For this reason, we supply a short self-contained
presentation.
% where this is clearly stated, 
%so we proceed with a short analysis of this.
We begin with the following lemma.

\begin{lem}\label{lem:algebra}
If $\M$ is logarithmically convex, the Carleman class $\calC_\M$ is an algebra.
\end{lem}
\begin{proof}
We may without loss of generality assume that $M_0=1$. Indeed, 
the classes are invariant under scaling of the sequence with the 
positive number $M_0^{-1}$. Under this assumption, together with
logarithmic convexity of $\M$, it follows that
\[
M_{j}M_{n-j}\leq M_n,\quad 0\leq j\leq n.
\]
Indeed, if $j\leq n-j$, since $M_0=1$,
\[
M_jM_{n-j}=M_{0}\left(\frac{M_1}{M_0}\cdots\frac{M_j}{M_{j-1}}\right)
\left(\frac{M_{n-j}}{M_{n-j+1}}
\cdots\frac{M_{n-1}}{M_n}\right)M_n.
\]
We next observe that
\[
\frac{M_{k}}{M_{k-1}}\leq\frac{M_{k+1}}{M_{k}}, k\geq 1,
\]
so we may estimate the product
\[
\left(\frac{M_1}{M_0}\cdots\frac{M_j}{M_{j-1}}\right)
\left(\frac{M_{n-j}}{M_{n-j+1}}
\cdots\frac{M_{n-1}}{M_n}\right)\leq 
\left(\frac{M_{n-j+1}}{M_{n-j}}\cdots\frac{M_{n}}{M_{n-1}}\right)
\left(\frac{M_{n-j}}{M_{n-j+1}}\cdots\frac{M_{n-1}}{M_n}\right)= 1.
\]
The claim now follows.

Next, let $f,g\in \calC_\M$, and suppose that  the estimates
\[
\lVert f^{(n)}\rVert_\infty\leq A_f^n M_n\quad \text{and}\quad
\lVert g^{(n)}\rVert_\infty\leq A_g^n M_n,
\]
hold for $n=0,1,2,\ldots$, where $A_f$ and $A_g$ are appropriate positive 
constants. Then, by the Leibniz rule, we have that
\[
\lVert (fg)^{(n)}\rVert_\infty \leq \sum_{j=0}^n
\binom{n}{j}\lVert f^{(j)}\rVert_\infty
\lVert g^{(n-j)}\rVert_\infty\leq\sum_{j=0}^n
\binom{n}{j} A_f^j A_g^{n-j}M_jM_{n-j}\leq (A_f+A_g)^n M_n,
\]
which proves that $fg\in \calC_\M$. By linearity, the conclusion extends
to all $f,g\in \calC_\M$.
\end{proof}

\begin{thm}\label{thm:p-ge-1}
When $1\le p<+\infty$ the $p$-Carleman class $\calC_\M^p$ 
is quasianalytic if and only if $\calC_\M$ is.
\end{thm}

\begin{proof}
We obtain the contrapositive statements. Assume that $\calC_\M$ is not 
quasianalytic. 
Then there exists functions $C_{\M}$ with compact support. 
Indeed, by the Denjoy-Carleman theorem it follows that
\[
\sum_{n\geq 1}\frac{M_{n-1}}{M_n}<+\infty,
\]
so $\Phi=\Phi_{1,\infty}$ constructed according to \eqref{eq:convol-def}, 
with $a_{n}=\frac{M_{n-2}}{M_{n-1}}$ and $a_1=\frac{1}{M_0}$ lies in the class
$\calC_\M$. As a consequence of the estimate
\[
\lVert \Phi^{(n)}\rVert_p\leq 
\lvert \supp \Phi\rvert ^{1/p}\lVert \Phi^{(n)}\rVert_\infty\leq 
\lvert \supp \Phi\rvert^{1/p}A_\Phi^nM_n,
\]
where $|E|$ denotes the length of the set $E\subset \R$, 
we find that $\Phi\in \calC_\M^p$ as well. In particular, the class 
$\calC_\M^p$ cannot be quasianalytic.

As for the other direction, we assume that $\calC_\M^p$ is not quasianalytic. 
Then there exists a function $f\in \calC_\M^p$ which vanishes to 
infinite degree at a point $a$ but is nontrivial at some other point $b$
( so that $f(b)\neq 0$). In view of translation as well dilatation invariance,
we may assume that $a=0$ and $b=\frac12$.
We now estimate the supremum norm on the unit interval $I=[0,1]$:
\[
\lvert f^{(n)}(x)\rvert=\lvert f^{(n)}(x)-f^{(n)}(a)\rvert \leq 
\| f^{(n+1)}\vert_{I}\|_1\leq \|f^{(n+1)}\vert_{I}\|_p\leq M_{n+1}, 
\]
using H{\"o}lder's inequality. We next form $g(x)=f(x)f(1-x)$. It turns out
that it follows from the argument used for the proof of 
Lemma~\ref{lem:algebra} that
\[
\lVert g^{(n)}\rVert_\infty\leq 2^nM_{n+1}.
\]
In particular, we have that $g\in\calC_{\M_1}$, where $\M_1=\{M_{n+1}\}_n$ is the
shifted weight sequence, so that $C_{\M_1}$ is not quasianalytic. 
Since a simple weight shift does not affect whether or not a Carleman class is 
quasianalytic, it follows that $\calC_\M$ is nonquasianalytic as well.
\end{proof}

%\begin{rem}
%For non-compact intervals, the same result should be possible to prove. 
%Before applying H{\"o}lder's inequality, one however has to produce 
%a compactly supported function $g$ from a 
%function $f\in C_\M^p$ which a priori only vanishes to infinite depth at one 
%point, e.g. if $f^{(n)}(0)=0$ for $n\geq0$. For such a function $f$, we 
%set $g(x)=f(x-x^2)$ for $0\leq x\leq 1$ and extend $g$ to be 
%zero outside $[0,1]$. 
%One has to verify that this construction does not cause the 
%$L^\infty$-norms to increase too much.
%\end{rem}

\subsection{Quasianalytic classes for $0<p<1$}\label{quasi-analytic}
 
Theorem~\ref{thm-quasianalytic} asserts that, for decay-regular sequences 
(see definition~\ref{def:reg-2}), $\calC_\M^{p,0}$ is quasianalytic if and 
only if $\calC_{\wtM}$ is. Here, the numbers $\twtM_{n}$ are the bounds that 
appeared in controlling $\|f^{(n)}\|_\infty$ by $\| f\|_\M$ where 
$f\in \S_{p,0,\M}^\circledast$, i.e.
\begin{equation}\label{eq:Ndef}
N_n=\prod_{j\geq 1}M_{n+j}^{p(1-p)^{j-1}}
\end{equation}
We denote by $\twtM_{n,\theta}$, the related numbers
\[
\twtM_{n,\theta}=\e^{\theta(1-p)^{-n}}\prod_{j=1}^{{+\infty}}M_{n+j}^{(1-p)^{j-1}p}.
\]

\begin{proof}[Proof of Theorem~\ref{thm-quasianalytic}] 
We begin with the assertion (i). From the logarithmic convexity of $\M$, 
we know  that $\wtM_\theta$ is also logarithmically convex. Indeed, 
\[
\frac{\twtM_{n-1,\theta}}{\twtM_{n,\theta}}
=\e^{-p\theta(1-p)^{-n}}\prod_{j\geq1}
\left(\frac{M_{n-1+j}}{M_{n+j}}\right)^{(1-p)^{j-1}p},
\]
and since
\[
\prod_{j\geq1}
\left(\frac{M_{n+j}}{M_{n+1+j}}\right)^{(1-p)^{j-1}p}
\leq \prod_{j\geq1}
\left(\frac{M_{n-1+j}}{M_{n+j}}\right)^{(1-p)^{j-1}p}
\]
it follows that
\[
\frac{\twtM_{n,\theta}}{\twtM_{n+1,\theta}}\leq \e^{-p^2(1-p)^{-n-1}}
\frac{\twtM_{n-1,\theta}}{\twtM_{n,\theta}}.
\]
Since the factor $\e^{-p^2(1-p)^{-n-1}}<1$ is of order $o(1)$ as $n\to+\infty$, 
we see in particular that for some $\rho<1$,
\[
\frac{\twtM_{n,\theta}}{\twtM_{n+1,\theta}}\leq 
\rho\frac{\twtM_{n-1,\theta}}{\twtM_{n,\theta}}.
\]
If we set $a_1=1/\twtM_1$, and $a_n=\twtM_{n-1}/\twtM_{n}$ for $n\geq 2$, 
and define $\Phi$ as in Section~\ref{ss:convol} by
\[
\Phi=\lim_{j\to+\infty} a_1^{-1} 1_{[0,a_1]}*\ldots a_j^{-1} 1_{[0,a_j]},
\]
then $\Phi$ is a compactly supported function, since 
$\{a_j\}_{j}\in \ell_1$.
By the estimate \eqref{eq-elest1.004}, we find that
\[
\lVert \Phi^{(n)}\rVert_p^p\leq \frac{2^n}{1-\rho}\frac{a_{n+1}}{a_1\cdots a_{n+1}}
=\frac{2^n}{1-\rho}M_n,
\]
so it follows that the required norm-estimate holds true.
Moreover, by \eqref{eq-elest1.002}, we see that
\begin{multline}
\label{eq-tameness-theta}
(1-p)^n\log\| \Phi^{(n)}\|_\infty =(1-p)^n\theta(1-p)^n + 
(1-p)^n\sum_{j=1}^{+\infty}(1-p)^{j}\log M_{n+j}
\\
=\theta+\sum_{j>n}(1-p)^j\log M_j=\theta+\mathrm{o}(1)
\end{multline}
as $n\to+\infty$, where the last step follows since the sum is the 
tail sum of the series which  converges to $\kappa(p,\M)<\infty$. 
It follows that $\Phi\in\calC_{\M}^{p,\theta}$, so this cannot be a 
quasianalytic class.

We continue with assertion (ii). 
We derive the contrapositive statement. If $\calC_{\M}^{p,0}$ is not 
quasianalytic, then there exists a nontrivial  $f\in \calC_\M^{p,0}$ such 
that $f^{(n)}(a)=0$ for all $n\geq0$ at some point $a\in \R$. 
The inequality \eqref{eq-basicineq1.9}, which in terms of the associated weight 
$\wtM$ states that
\[
\lVert f^{(n)}\rVert_\infty\leq N_{n,0} \lVert f\rVert_{p,\M},
\]
now shows that $f\in C_{\wtM}$, so that $C_{\wtM}$ cannot be a quasianalytic 
class.

Finally, as for the remaining implication, we assume that $C_{\wtM}$ is 
non-quasianalytic. We single 
out the two cases when $\sum_n M_{n-1}/M_n=+\infty$ 
and when $\sum_n M_{n-1}/M_n<{+\infty}$. In the latter case, there are 
nontrivial compactly supported functions in $\calC_\M$. 
Take one such function $f$, and set $B:=\lvert \supp f\rvert$. We now claim 
that $f\in \calC_\M^{p,0}$. First note that since $\kappa(p,\M)$ is finite,
\[
(1-p)^{n}\log \lVert f^{(n)}\rVert_\infty\leq (1-p)^n\log (C_f A_f^n M_n)=
\mathrm{o}(1),
\] 
so $f$ is $(p,0)$-tame. Next, it is clear that
\[
\lVert f^{(n)}\rVert_p\leq B^{1/p}C_f A_f^nM_n,
\]
since $\lVert f^{(n)}\rVert_\infty\leq C_f A_f^{n}M_n$ and $f$ has compact 
support. 

Hence we may assume that $\sum_n M_{n-1}/M_n=+\infty$. In that case, our 
regularity assertion tells us that
\[
M_{n+1}/M_n\leq \epsilon^{-n},
\]
for some $\epsilon>0$.
As before, we set $a_1=1/N_0$ and $a_n=N_{n-2}/N_{n-1}$ for $n\geq 2$,
where $\wtM=\{\twtM_n\}_n$ is again defined by \eqref{eq:Ndef}.
It follows that
\[
a_{n}^{-1}=\frac{\twtM_{n-1}}{\twtM_{n-2}}
=\prod_{j\geq1}\left(\frac{M_{n-1+j}}{M_{n-2+j}}\right)^{p(1-p)^{j-1}}
\leq \prod_{j\geq1}\epsilon^{-(j+n)pq^{j-1}}\leq C\epsilon^{-n},
\]
for some constant $C<+\infty$. Then, if $R$ is large enough, it follows that 
$R^na_n\geq \lVert \{a_j\}\rVert_{\ell^1}$,
and thus that $\sum_{j>n} a_j \leq R^na_n$.
By the same methods as in case (1), it follows that there exists a function 
$\Phi\in \calC_\M^{p,0}$ with compact support, which 
thus cannot be a quasianalytic class.
\end{proof}

We conclude with our last remaining theorem.

\begin{proof}[Proof of Theorem~\ref{theta-thm}]
We are given an increasing sequence $\M$ such that $\kappa(p,\M)<+\infty$, 
and numbers $0\leq \theta<\theta'$. First note that the tameness requirement 
is stable 
under the passage from $\S_{p,\theta,\M}^{\circledast}$ to $W^{p,\theta}_\M$, 
when $\M$ is bounded away from zero: Let $f\in W^{p,\theta}_\M$. 
Taking logarithms in \eqref{eq-basicineq1.9}, we have the estimate
\[
\log \lVert f^{(n)}\rVert_{\infty}\leq\theta(1-p)^{-n}+ \log \lVert f\rVert_{p,\M} 
+ p(1-p)^{-n-1}\kappa(p,\M)-\sum_{j=1}^n(1-p)^{j-n-1}p\log M_{j},
\]
so
\begin{multline}\label{eq:tame-conv}
(1-p)^n\log \lVert f^{(n)}\rVert_\infty\leq \theta+(1-p)^n\log 
\lVert f\rVert_{p,\M} 
\\+p(1-p)^{-n-1}\kappa(p,\M)-p(1-p)^{-n-1}\sum_{j=1}^n(1-p)^j\log M_j
\end{multline}
By Remark~\ref{rem:001} (b), the partial sums in \eqref{eq:tame-conv} 
converge to $\kappa(p,\M)$ as $n\to+\infty$, which shows that
\[
\limsup_{n\to+\infty} (1-p)^n\log \lVert f^{(n)}\rVert_{p,\M}\leq \theta.
\]
Hence, in order to show that the containment $W_\M^{p,\theta'}\supset 
W_\M^{p,\theta}$ is strict, it suffices to exhibit a function 
$f\in W_\M^{p,\theta'}$ such that \eqref{eq-tameness} fails.
To this end, we consider the function 
$\Phi=\Phi_{\theta'}\in \S_{p,\theta',\M}^\circledast$ 
which was constructed in the proof of Theorem~\ref{thm-quasianalytic}, 
part (i), which is in violation of quasianalyticity. 
According to \eqref{eq-tameness-theta}, we have that
$(1-p)^n\log \lVert \Phi^{(n)}\rVert_\infty\to \theta'$ as $n\to+\infty$, 
which completes the proof.
\end{proof}

\section{Concluding remarks}
\label{sec-remarks}
A number of questions remain to be investigated. For instance, we would 
like to better understand the space $W_\M^{p,\theta}$ in the uncoupled regime. 
Note that when the $p$-characteristic $\kappa(p,\M)$ is infinite, we have 
a strict inclusion $W_\M^{p,\theta}\subset W_\M^{p,\theta'}$, $W_\M^{p,\theta}\ne 
W_\M^{p,\theta'}$ 
when $\theta<\theta'$ (See Proposition~\ref{theta-thm}). When 
$\kappa(p,\M)=+\infty$ we do not know whether this is the case. 
It is conceivable that these spaces are so large that the $\theta$-dependence 
is lost. 
We note that Proposition~\ref{prop:inclusion-tail-decay}, gives the inclusions
\[
c_0(L^p,\M)\subset \bp W_\M^{p,\theta}\subset \ell^\infty(L^p,\M).
\]

Another issue is whether we can drop some of the regularity assumptions
in Theorems~\ref{thm-uncoupling} and \ref{thm-quasianalytic}. 

In addition, it appears to be of interest to study the case when the maximal 
class $\S_{p,\theta,\M}^\circledast$ is replaced with some other smaller 
class of test functions, for which \eqref{eq-tameness} remains valid. 
For instance, we might consider the Hermite class
\[
\S_{p,\M}^{\mathrm{Her}}=\big\{f(x)=\e^{- x^2}p(x): \,
p \text{\: a polynomial such that }\, \| f\|_{p,\M}<+\infty\big\}
\]
which was mentioned earlier, or, in the case of the unit circle, 
the class $\mathscr{P}_{\text{trig}}$ of trigonometric polynomials,
and ask whether the corresponding $L^p$-Carleman spaces undergo the same 
phase transitions.

\end{document}